\documentclass[a4paper,12pt]{amsart}
\usepackage{amsmath,amssymb}
\usepackage{tikz}
\usepackage[normalem]{ulem}

\allowdisplaybreaks[2]

\numberwithin{equation}{section}

\newtheorem{thm}{Theorem}[section]
\newtheorem{prop}[thm]{Proposition}
\newtheorem{lemma}[thm]{Lemma}
\newtheorem{cor}[thm]{Corollary}

\theoremstyle{definition}
\newtheorem{defn}[thm]{Definition}
\theoremstyle{remark}
\newtheorem{rmk}[thm]{Remark}
\newtheorem{ex}[thm]{Example}

\newcommand{\C}{\mathbb{C}}
\newcommand{\N}{\mathbb{N}}

\newcommand{\Z}{\mathbb{Z}}

\newcommand{\go}{G^{(0)}}
\newcommand{\KP}{\operatorname{KP}}

\newcommand{\supp}{\operatorname{supp}}

\newcommand{\lsp}{\operatorname{span}}
\newcommand{\interior}{\operatorname{Int}}
\newcommand{\iso}{\operatorname{Iso}}
\newcommand{\image}{\operatorname{im}}

\begin{document}

\title[Ideals of Steinberg algebras]
{Ideals of Steinberg algebras of strongly effective groupoids, with applications to Leavitt path algebras}

\author{Lisa Orloff Clark}
\address{Lisa Orloff Clark\\
    Department of Mathematics and Statistics\\
    University of Otago\\
    PO Box 56\\
    Dunedin 9054\\
    New Zealand}
\email{lclark@maths.otago.ac.nz}

\author{Cain Edie-Michell}
\address{Cain Edie-Michell\\
Mathematical Sciences Institute\\
The Australian National University\\
Union Lane\\
Canberra ACT 2601\\
Australia}
\email{cain.edie-michell@anu.edu.au}

\author{Astrid an Huef}
\address{Astrid an Huef\\
    Department of Mathematics and Statistics\\
    University of Otago\\
    PO Box 56\\
    Dunedin 9054\\
    New Zealand}
\email{astird@maths.otago.ac.nz}

\author{Aidan Sims}
\address{Aidan  Sims\\
    School of Mathematics and Applied Statistics\\
    University of Wollongong\\
    NSW 2522\\
    Australia}
\email{asims@uow.edu.au}


\date{26 January 2016}

\begin{abstract}
We consider the ideal structure of Steinberg algebras over a commutative ring with identity. We
focus on Hausdorff groupoids that are strongly effective in the sense that their
reductions to closed subspaces of their unit spaces are all effective. For such a
groupoid, we completely describe the ideal lattice of the associated Steinberg algebra
over any commutative ring with identity. Our results are new even for the special case of
Leavitt path algebras; so we describe explicitly what they say in this context, and give
two concrete examples.
\end{abstract}

\thanks{This research was supported by Marsden grant 15-UOO-071 from the Royal Society of New Zealand and by Discovery Project grant DP150101598 from the Australian Research Council.  Part of this research was completed while the fourth author was attending the workshop \emph{Classification of operator algebras: complexity, rigidity, and dynamics} at the Mittag-Leffler Institute.} 

\maketitle

\section{Introduction}
Leavitt path algebras over a field have been studied intensively
since their independent introduction, around 2005,  by Abrams--Aranda-Pino in \cite{AbrPino} and
Ara--Moreno--Pardo in \cite{AMP}.  One of the earliest questions asked about these algebras was what the ideals look like. The lattice of ideals is now completely understood, see, for example, \cite[Theorem~2.8.10]{AbramsAraMolina-book} or \cite[Theorem~11]{Abrams}.
Work on the ideal structure of the Leavitt path algebra and its irreducible representations continues. See for example the recent papers on the generators of ideals \cite{Ranga2014}, prime and primitive ideals \cite{APS2009, Ranga2013, EsinKanuniRanga}, two-sided chain conditions \cite{ABCK2012}, and  on irreducible representations \cite{Chen, AraRanga, HazratRanga}.

In 2011, Tomforde  went on to consider Leavitt path algebras over commutative rings
$R$ with identity in \cite{Tomforde:JPAA11}, and again considered the ideal structure.
Things are more complicated in this setting because the ideal structure of the ring $R$
has an effect on the ideal structure of the Leavitt path algebra. Tomforde sidestepped
this issue by considering only the ``basic'' ideals which are, roughly speaking, the
ideals that contain a scalar multiple of a generator if and only if they contain the
generator itself, and are therefore insensitive to the ideal structure of $R$.  
The structure of the basic ideals in  Leavitt path algebra has recently been reconsidered by Larki in \cite{Larki} for more general graphs than were allowed in \cite{Tomforde:JPAA11}. Larki also studies the prime and primitive ideals, and this involves non-basic ideals.

In this paper, we investigate the basic and non-basic ideal structure of a large
class of Steinberg algebras.  The Steinberg algebras, introduced independently in \cite{Steinberg:AM10} and in \cite{CFST},  are associated to ample groupoids. They include the Kumjian--Pask algebras of higher-rank graphs introduced in \cite{ACaHR}, which in turn include the Leavitt path algebras.  The  advantage of working with the more general Steinberg algebras is that this brings into play, in the algebraic setting, powerful techniques from Renault's theory of groupoid $C^*$-algebras. Indeed, Renault's theory previously played a fundamental role in the development of the theory of graph $C^*$-algebras and their analogues.

Groupoid models in $C^*$-algebra theory are particularly well-suited to answering
questions about ideal structure \cite{Renault:JOT91}. We focus on groupoids $G$ which are
strongly effective in the sense that in every reduction of $G$ to a closed invariant
subspace of its unit space $G^{(0)}$, the interior of the isotropy consists only of
units. This reduces to Condition~(K) for graphs and to ``strong aperiodicity'' for
higher-rank graphs. (This is folklore, but we provide a proof in Corollary~\ref {cor-se-iff-sa}.) Our results provide a complete description of the lattice of ideals in the
Steinberg algebra of such a groupoid. Since these results are new even for Leavitt path
algebras, and hence also for Kumjian--Pask algebras, we give an explicit account of what
our main theorem  says in these special cases.

We start in  Section~\ref{sec:basic} by analysing the basic ideals of the Steinberg algebra of a strongly
effective groupoid . We find that the ideals are indexed by the
open invariant subsets of the unit space as expected. When $R$ is a field, every
ideal is a basic ideal. Thus we can  draw some conclusions 
about the ideals in Steinberg algebras over fields, and the relationship of these to the ideals
of the corresponding groupoid $C^*$-algebra, at least when the groupoid $G$ is amenable.

In Section~\ref{sec:nonbasic} we build on our analysis of basic ideals to describe 
all the ideals in the Steinberg algebra. The extra ideals  arising from the ideals of the ring $R$ are encoded by functions $\pi$, satisfying a consistency condition relating nesting of ideals in $R$ to nesting of subsets of the unit space, from the collection of open invariant subsets of $G^{(0)}$ to the set
$\mathcal{L}(R)$ of ideals of $R$.

Containment of ideals in the Steinberg algebra is encoded by a very natural partial order
on the functions $\pi$ described in the preceding paragraph. So in principle the lattice
structure on the set of ideals is explicitly described in terms of the functions $\pi$.
However, it is difficult to describe the join operation on functions $\pi$ that
corresponds to addition of ideals of the Steinberg algebra. In Section~\ref{sec:lattice},
we introduce an alternative characterisation of the ideals in the Steinberg algebra in
terms of functions $\rho : G^{(0)} \to \mathcal{L}(R)$ that are continuous with respect
to a suitable topology on $\mathcal{L}(R)$. This allows us to describe the join 
and meet operations quite naturally.

Finally, in Section~\ref{sec:Lpa/KPa}, we translate our results into the language of
Leavitt path algebras and Kumjian--Pask algebras. Here, the ideals are parameterised by
functions from the collection of saturated hereditary subsets of the vertex set of the
graph into the set of ideals of $R$, again satisfying a suitable consistency condition;
or alternatively by continuous functions from the infinite-path space of the graph to the
ideal space $\mathcal{L}(R)$. We detail the content of our theorems for two concrete
examples of graphs, each emphasising the advantages of one of these two
parameterisations.

\section{Preliminaries}\label{sec:prelims}
We use the groupoid conventions of \cite{CE}.
Let $G$ be a groupoid. A subset $U$ of the unit space  $\go$ of $G$ is \emph{invariant}
if $s(\gamma) \in U$ implies $r(\gamma) \in U$; equivalently, 
\[
    r(s^{-1}(U)) = U = s(r^{-1}(U)).
\]
 Given $V \subseteq\go$, we write $[V]$ for the smallest invariant subset of $\go$ containing $V$. Thus
\[
    [V] = r(s^{-1}(V)) = s(r^{-1}(V)).
\]
We use the standard
notation from \cite[page~6]{Renault1980} where $G_u = \{\gamma \in G : s(\gamma) = u\}$, $G^u =
\{\gamma \in G : r(\gamma) = u\}$ and $G^u_u = G_u \cap G^u$ for each unit $u \in
G^{(0)}$. The \emph{isotropy groupoid} of $G$ is
\[
    \iso(G) := \{ g\in G: s(\gamma) = r(\gamma)\} = \bigcup_{u \in G^{(0)}} G^u_u.
\]

Let $U$ be an invariant subset of $\go$. We write $G_U := s^{-1}(U)$, 
and then $G_U$ coincides with the \emph{restriction}
\[
 G|_U:=\{\gamma \in G : s(\gamma), r(\gamma) \in U\}.
\]
of $G$ to $U$. This $G_U$ is a groupoid with unit space $U$.

For subsets $W,V \subseteq G$, we define $WV := \{\gamma\eta : \gamma \in W, \eta \in V,
s(\gamma) = r(\eta)\}$.

Now let $G$ be a topological groupoid. 
A  subset $B$ of $G$ is a \emph{bisection} if the source and range maps restrict 
to homeomorphisms on $B$; for an open set to be a bisection we require the source and range maps to restrict 
to homeomorphisms onto open subsets of $\go$. 
Then $G$ is called \emph{ample} if $G$ has a basis of
compact open bisections. In this paper, we only consider ample Hausdorff  groupoids.

An ample Hausdorff groupoid $G$
is \emph{effective} if the interior of $\iso(G)$  is just $\go$.
It follows that when $G$ is effective, if $B$ is a nonempty compact open bisection such that $B \subseteq G\setminus\go$, then $B \setminus \iso(G) \neq \emptyset$.

When $G$ is second countable, $G$ is effective if and only if it is topologically
principal in the sense that $\{u \in \go: G^u_u = \{u\}\}$  is dense
in  $\go$ (see \cite[Proposition~3.6]{Renault:IMSB08}). Our results apply to
groupoids $G$ that are not second countable, so for us the two conditions are, in
general, different.

\begin{defn}
A groupoid $G$ is \emph{strongly effective} if for every nonempty closed
invariant subset $V$ of $\go$, the groupoid $G_V$ is effective.
\end{defn}

If $G$ is strongly effective, then it is effective because $\go$ is a closed
invariant set. If $G$ is second countable, then so is $G_V$ for every invariant subset $V$ of $\go$, and so $G$ is
strongly effective if and only if it is essentially principal in the sense of
\cite[Chapter~2, Definition~4.3]{Renault1980}.

Let $G$ be  an ample Hausdorff  groupoid   and $R$ a commutative ring with identity. We write
$A_R(G)$ for the \emph{Steinberg algebra} of all locally constant, compactly supported
functions $f : G \to R$, equipped with the convolution product. As a set, $A_R(G)$
is the $R$-linear span \[\lsp_R\{1_B : B\text{ is a compact open bisection}\}.\]

For $f\in A_R(G)$, the set $\{\gamma\in G : f(\gamma) \not= 0\}$ is a finite
union of compact open sets, and so is itself compact and open.  Since compact subsets of
a Hausdorff space are closed, we have
\[
    \supp(f) := \overline{\{\gamma \in G : f(\gamma) \not= 0\}} =  \{\gamma \in G : f(\gamma) \not= 0\}.
\]
Under the convolution product on $C_c(G)$, for $f,g\in A_R(G)$ we have $\supp(f * g) \subseteq \supp(f)
\supp(g)$.

\section{Basic ideal structure}\label{sec:basic}

Throughout, $G$ is an ample Hausdorff  groupoid and $R$ is a commutative ring with identity.

When the coefficient ring $R$ is not a field, the ideal structure of $A_R(G)$ depends on
the ideal structure of $R$.   For example, if $G = \{e\}$ is the trivial group, then the Steinberg algebra
$A_R(G)$ is isomorphic to $R$ as an $R$-algebra, and then the ideals of $A_R(G)$ are precisely the
ideals of $R$.
An ideal $I$ of $A_R(G)$ is a \emph{basic ideal} if 
\[
 \text{$K$ a compact open  subset of $\go$}, 0\neq r\in R \text{ and }   r1_K \in I \quad \implies \quad 1_K \in I.\]
 When $G = \{e\}$, the only nonzero basic ideal is $R$ itself. In
general, the basic ideals are the ones that reflect the structure of $G$ alone, and do
not reflect the structure of $R$; we expect  the basic-ideal structure to be
independent of $R$.  Basic ideals of $A_R(G)$ were
introduced by the first two authors in \cite{CE}, and they generalise the basic ideals
of a Leavitt path algebra studied by Tomforde in \cite{Tomforde:JPAA11}. 

The first step in studying the ideal structure of $A_R(G)$  is to study the basic ideals.
By \cite[Theorem~4.1]{CE}, if $G$ is an ample Hausdorff groupoid, then $A_R(G)$ has no proper basic ideals if and only if $G$ is effective and minimal. In this paper we consider groupoids that are strongly effective (hence
effective) but not minimal. The main result of this section is the following.

\begin{thm} \label{thm:basicideal}
Let $G$ be an ample Hausdorff groupoid, and let  R be a commutative ring with identity.  Then $G$ is strongly effective if and only if 
\[
    U \mapsto I_U:=\{f \in A_R(G) : \supp f \subseteq G_U\}
\]
is a lattice isomorphism from the  open invariant subsets of $\go$ onto the  basic ideals of $A_R(G)$. 
\end{thm}

Before proving Theorem~\ref{thm:basicideal}, we need to establish some helper results.

\begin{lemma}
\label{lem:IU} Let $G$ be an ample Hausdorff groupoid,  let  R be a commutative ring with identity and let $U \subseteq \go$ be an
open invariant subset. Then $I_U$ is a basic ideal in $A_R(G)$.
\end{lemma}
\begin{proof}
The set $I_U$ is closed under addition and scalar multiplication. To see that
$I_U$ is an ideal, fix $f \in I_U$ and $g \in A_R(G)$.
Let $\alpha \notin G_U$ and $\beta \in G^{r(\alpha)}$. Then
$s(\beta^{-1}\alpha) = s(\alpha) \not\in U$, and hence $s(\beta) \not\in U$ because $U$
is invariant. Hence $f(\beta) = 0$. Thus
\[
    (f*g)(\alpha) = \sum_{\beta \in G^{r(\alpha)}} f(\beta)g(\beta^{-1}\alpha) = 0.
\]
So $f*g \in I$.  A similar argument gives $g*f \in I$. That $I_U$ is basic follows
immediately from its definition.
\end{proof}

\begin{prop}\label{prop:injective} Let $G$ be an ample Hausdorff groupoid and let  R be a commutative ring with identity. 
Then
$U \mapsto I_U$
is an injective lattice morphism from the   open invariant subsets of $\go$ to the basic ideals of $A_R(G)$. 
\end{prop}

\begin{proof}
We first prove that for open invariant subsets $U,V$ of $\go$, we have $I_U \subseteq I_V$ if and only if
$U \subseteq V$. Suppose that $I_U \subseteq I_V$, and fix $u \in U$. Choose a compact open neighbourhood $K$ of
$u$ such that $K \subseteq U$. Then $1_K \in I_U \subseteq I_V$, giving $u \in K
\subseteq V$. Hence $U \subseteq V$.  Conversely, if $U\subseteq V$, then $G_U\subseteq G_V$, and hence $I_U\subseteq I_V$.

It follows immediately that $I_U = I_V$ implies $U = V$, so $U \mapsto I_U$ is injective.

We will show that \[I_{U \cap V} = I_U \cap
I_V\quad\text{and}\quad I_{U \cup V} = I_U + I_V\] Since the set of open invariant subsets of $\go$ (with set inclusion, intersection and union) forms a lattice, it will then follow that $\{I_U: \text{$U$ is an open invariant subset of $\go$}\}$ is a lattice  (with set inclusion, intersection and $+$),  and that $U\mapsto I_U$ is a lattice morphism. 

Since $G_U \cap G_V = s^{-1}(U) \cap s^{-1}(V) =
s^{-1}(U \cap V) = G_{U \cap V}$, we have
\[
I_U \cap I_V
    = \{f : \supp(f) \in G_U \cap G_V\}
    = \{f : \supp(f) \in G_{U \cap V}\} = I_{U \cap V}.
\]
If $f \in I_U + I_V$, say $f = f_U + f_V$, then
\[
\supp(f)
    \subseteq \supp(f_U) \cup \supp(f_V)
    \subseteq G_U \cup G_V
    = s^{-1}(U) \cup s^{-1}(V)
    = s^{-1}(U \cup V)
    = G_{U \cup V}.
\]
This gives $I_U + I_V \subseteq I_{U \cup V}$. 

For the reverse containment, suppose that $f \in I_{U \cup V}$. So $\supp(f) \subseteq
s^{-1}(U \cup V)$. The set $K_U := \supp(f) \setminus s^{-1}(V) \subseteq s^{-1}(U)$ is
compact because it is a closed subset of $\supp(f)$, and similarly $K_V := \supp(f)
\setminus s^{-1}(U)$ is a compact subset of $s^{-1}(V)$. Let $u\in K_U$. Since $\supp(f) \cap s^{-1}(U)$ is open,
and since $G$ is ample, we can find a compact open neighbourhood $N_u$ of $u\in N_u \subseteq\supp(f) \cap s^{-1}(U)$. By taking the union of a finite subcover
of the cover $\{N_u : u \in K_U\}$ of $K_U$, we obtain a compact open subset $K'_U$ of
$\supp(f)$ such that $K_U \subseteq K'_U \subseteq s^{-1}(U)$. Let $K'_V := \supp(f)
\setminus K'_U$. Then $\supp(f) = K'_U \sqcup K'_V$ with $K'_U \subseteq s^{-1}(U)$ and
$K'_V \subseteq s^{-1}(V)$. Since $K'_U$ and $K'_V$ are compact and open, we obtain
locally constant functions $f_U$ and $f_V$ by setting $f_U(\gamma) := 1_{K'_U}(\gamma) f(\gamma)$ and
$f_V(\gamma) := 1_{K'_V}(\gamma) f(\gamma)$. By construction, $f_U \in I_U$ and $f_V \in I_V$, and so $f
= f_U + f_V  \in I_U + I_V$. Thus $I_U + I_V = I_{U \cup V}$. As discussed above, this proves the proposition.
\end{proof}

\begin{lemma}
 \label{lem:shouldadone}
Let $G$ be an ample Hausdorff groupoid and let  R be a commutative ring with identity. 
Suppose that $G$ is  not effective. Then there is
a nonzero basic ideal $I$ of $A_R(G)$ such that $I \cap A_R(\go) = \{0\}$.
\end{lemma}
\begin{proof}
Let $\mathbb{F}_R(\go)$ denote the free $R$-module generated by a copy of $\go$; to
reduce confusion, we shall write $\delta_u$ for the spanning element of
$\mathbb{F}_R(\go)$ corresponding to $u \in \go$. Let
$\operatorname{End}(\mathbb{F}_R(\go))$ denote the $R$-algebra of endomorphisms of
$\mathbb{F}_R(\go)$.  By applying \cite[Proposition~4.2(2)]{CE} to the $G$-invariant set $G^{(0)}$ there is a
homomorphism $\pi :A_R(G) \to \operatorname{End}(\mathbb{F}_R(\go))$ such that
\[
    \pi (f)\delta_u = \sum_{\gamma \in G_u} f(\gamma)\delta_{r(\gamma)}.
\]
Since $G$ is not effective, $\interior(\iso(G))\setminus G^{(0)}$ is nonempty. Since $\interior(\iso(G))\setminus G^{(0)}$ is open, there exists a compact open bisection $B \subseteq \interior(\iso(G) )\setminus \go$. If $u\in s(B)$ then  $\pi (1_B)\delta_u=\delta_u=\pi(1_{s(B)})\delta_u$, and both are $0$ otherwise. Now $0\neq 1_B - 1_{s(B)} \in \ker\pi$.
Thus $\ker\pi$ is a nonzero ideal, and it is basic by \cite[Lemma~4.5]{CE}.

We will show that
$\ker\pi \cap A_R(\go) = \{0\}$, and this proves the lemma.
Let $f \in A_R(\go) \setminus \{0\}$.  Fix $u \in \go$ such that $f(u) \not= 0$. Then 
$0 \not= f(u)\delta_u = \pi (f)\delta_u$, and so $f \notin \ker\pi$. Thus $\ker\pi \cap A_R(\go) = \{0\}$.
\end{proof}

If $G$ is an effective ample Hausdorff groupoid, then every nonzero ideal $I$ of $A_R(G)$ has nonzero intersection with $A_R(\go)$  by \cite[Proposition 3.3]{St2}. Combining this with Lemma~\ref{lem:shouldadone} gives the following corollary.

\begin{cor}
Let $G$ be an ample Hausdorff groupoid and let  R be a commutative ring with identity. Then $G$ is effective if and only if every
nonzero ideal $I$ of $A_R(G)$ has nonzero intersection with $A_R(\go)$.
\end{cor}

Suppose that $U$ is an open invariant subset of $\go$ and let $D := \go\setminus U$.
Since $U$ is open, there is a function $i_U : A_R(G_U) \to A_R(G)$ such that
\[
i_U(f)(\gamma) =
    \begin{cases}
        f(\gamma) &\text{ if $\gamma \in U$}\\
        0 &\text{ otherwise.}
    \end{cases}
\]
Likewise, since $D$, and hence $G_D$, is closed, restriction of functions gives  a
function $q_U : A_R(G) \to A_R(G_D)$.

\begin{lemma}\label{exact seq} Let $G$ be an ample Hausdorff groupoid and let  R be a commutative ring with identity.
Let $U$ be an open invariant subset of $\go$, and $D:=\go\setminus U$. The functions $i_U:
A_R(G_U)\to A_R(G)$ and $q_U:A_R(G)\to A_R(G_D)$ are $^*$-homomorphisms, and the sequence
\[
    0 \longrightarrow A_R(G_U) \overset{i_U}{\longrightarrow} A_R(G) \overset{q_U}{\longrightarrow} A_R(G_D) \longrightarrow 0
\]
is exact. Further, $I_U = i_U(A_R(G_U))$, and 
\[
I_U = \lsp_R\{1_B : B \subseteq G\text{ is a compact open bisection with } s(B) \subseteq U\}.
\]
\end{lemma}
\begin{proof}
Since $U$ is invariant, $i_U: A_R(G_U)\to A_R(G)$ is a homomorphism, and since $D$ is
invariant, $q_U$ is also a homomorphism. It is clear that $i_U$ is injective. To see that
$q_U$ is surjective, fix a compact open subset $K$ of $G_D$. Since $K$ is also compact in $G$, and $G$
is ample, we can find a finite cover $\bigcup_{L \in F} L$ of $K$ by mutually
disjoint compact open subsets of $G$. Then $1_K = q_U(\sum_{B \in F} 1_L)$. Since
$A_R(G_V)$ is spanned by the $1_K$ it follows that $q_U$ is surjective. By definition of $i_U$ and $q_U$ it is clear that $\image i_U\subseteq \ker q_U$. For the reverse
containment, take $f \in \ker q_U$. Write $f = \sum_{B \in F} r_B 1_B$ where $F$ is a
collection of mutually disjoint bisections of $G$ and the $r_B$ are all nonzero. Since
$q_U(f) = 0$, each $B \in F$ is contained in $G_U$, and so is a compact open subset of
$G_U$. So we can define $f_0 \in A_R(G_U)$ by $f_0 = \sum_{B \in F} r_B 1_B$, and we have
$i_U(f_0) = f$ by construction. Finally, we have
\begin{align*}
I_U &= i_U(A_R(G_U))
    = \lsp_R\{i_U(1_B): B \text{ is a compact open bisection of } G_U\}\\
    &= \lsp_R\{1_B : B \subseteq G \text{ is a compact open bisection with } s(B) \subseteq U\}.\qedhere
\end{align*}
\end{proof}

\begin{proof}[Proof of Theorem~\ref{thm:basicideal}.]  By Proposition~\ref{prop:injective},
$U \mapsto I_U$ is an injective lattice morphism.  So it remains to prove that $G$ is
strongly effective if and only if  $U \mapsto I_U$ is surjective.

Suppose that $G$ is not strongly effective. There exists a nonempty closed invariant
subset $V$ of $\go$ such that $G_V$ is not effective. By Lemma~\ref{lem:shouldadone},
there is a nonzero basic ideal $I$ of $A_R(G_V)$ which has zero intersection with
$A_R(G_V^{(0)})$. 
Let
\[
    J := \{f \in A_R(G) : f|_{G_V} \in I\}.
\]
If $V=\go$, then $J=I$. If $V\neq \go$, then $I_{\go\setminus V}$ is a nonzero ideal of $A_R(G)$ contained in $J$. In either case, $J$ is a nonzero ideal of $A_R(G)$.

To see that $J$ is a  basic ideal, suppose that $K
\subseteq \go$ is a compact open subset of $\go$ and $0\neq r\in R$ with $r 1_K \in J$.
Then $r 1_K|_{G_V}=r 1_{K \cap V} \in I$. Since $I$ is basic, $1_K|_{G_V}=1_{K \cap V}
\in I$, and hence $1_{K} \in J$. Thus $J$ is basic.

To see that $J$ is not of the form $I_U$, fix a nonempty open invariant $U \subseteq
\go$. First suppose that $U \cap V = \emptyset$. Fix a nonzero element $g \in I$.
Lemma~\ref{exact seq} shows that $q_V : A_R(G) \to A_R(G_V)$ is surjective, so there
exists $f \in A_R(G)$ such that $f|_{G_V} = g$, and so $f \in J$. Since $U \cap V =
\emptyset$, we have $h|_{G_V} = 0$ for all $h \in I_U$, and we conclude that $f \in J
\setminus I_U$. On the other hand, if $U \cap V \not= \emptyset$, then any nonempty
compact open subset $K \subseteq U \cap V$ satisfies $1_K\in I_U \cap A_R(G_V^{(0)})$.
Since $J \cap A_R(G_V^{(0)}) = \{0\}$, this implies $J \not= I_U$. Thus $U \mapsto I_U$
is not surjective.

Conversely, suppose that $G$ is strongly effective.  We just have to show that $U \mapsto
I_U$ is surjective. Let $I$ be a nonzero basic ideal in $A_R(G)$. Define
\[
    D:= \bigcap_{f \in I, \supp f \subseteq \go} f^{-1}(0)
        \quad \text{ and } \quad
    U:= \go \setminus D = \bigcup_{f \in I, \supp f \subseteq \go} f^{-1}(R \setminus \{0\}).
\]
We will show that $U$ is an open invariant subset of $\go$ and that $I = I_U$.

To see that $U$ is invariant, let $u\in U$ and choose $\gamma$ such that $s(\gamma)=u$.
We must show that $r(\gamma) \in U$. Fix $f \in  I$ such that $\supp f \subseteq \go$ and
$f(s(\gamma)) \neq 0$. Let $B$ be a compact open bisection containing $\gamma$. A
calculation shows that
\begin{align*}
 1_B*f*1_{B^{-1}}(\zeta)
 &=\sum_{\{\eta\in B: r(\eta)=r(\zeta)\}} f(s(\gamma))1_{B^{-1}}(\eta^{-1}\zeta)\\
 &=\begin{cases}
 f(s(\eta))&\text{if there exists unique $\eta\in B$ such that $\zeta=r(\eta)$}\\
 0&\text{else}.
 \end{cases}
\end{align*}
In particular, $1_B*f*1_{B^{-1}}(r(\gamma))=f(s(\gamma))\neq 0$ and $\supp(
1_B*f*1_{B^{-1}})\subseteq\go$. Thus $r(\gamma)\notin D$, and hence $r(\gamma)\in U$.
Thus $U$ is invariant, and it is open because it is a union of open sets $f^{-1}(R
\setminus \{0\})$.

Now we will show that $I=I_U$. For the $\subseteq$ direction, recall from
Lemma~\ref{exact seq} that $I_U = \ker q_U$, and so $q_U$ induces an isomorphism $\tilde
q_U:A_R(G)/I_U\to A_R(G_D)$. By definition of $I_U$, we have $I \cap A_R(\go) = \{f \in
A_R(\go) : \supp(f) \subseteq U\} = I_U \cap A_R(\go)$. Thus
\begin{align*}
\tilde q_U(I+I_U)\cap A_R((G_D)^{(0)})
&=\tilde q_U\big( (I\cap A_R(\go) )+I_U\big)\\
&=q_U(I \cap A_R(\go))\\
&=q_U(I_U \cap A_R(\go))=\{0\}.
\end{align*}
Since $G$ is strongly effective, $G_D$ is effective, and \cite[Proposition 3.3]{St2}
gives $\tilde{q}_U(I+I_U) = \{0\}$. Thus $I \subseteq  I_U$.

For the $\supseteq$ direction, we first claim that if $B$ is a compact open bisection and
$s(B) \subseteq U$, then $1_{s(B)} \in I$. To see this, observe that $s(B)$ is a compact
open subset of $U$. By definition of $U$, for each $u \in U$ there is an element $f_u \in
I$ such that $\supp(f_u) \subseteq \go$ and $f_u(u) \not= 0$. This $f_u$ is locally
constant, so $V_u := f^{-1}(f(u))$ is a compact open neighbourhood of $u$ in $\go$ and
$f|_{V_u} = f(u)1_{V_u}$. Since $I$ is an ideal, we deduce that $f(u)1_{V_u} = f *
1_{V_u}$ belongs to $I$. Since $I$ is a basic ideal, we deduce that $1_{V_u} \in I$. Now
the $V_u$ cover $s(B)$, which is compact, so we can write $s(B)$ as a finite union $s(B)
= V_{u_1} \cup \dots \cup V_{u_n}$. Putting $W_i := V_{u_i} \setminus \bigcup^{i-1}_{j=1}
V_{u_j}$ we obtain pairwise disjoint compact open sets that cover $s(B)$, and each
$1_{W_i} = 1_{W_i} * 1_{V_{u_i}} \in I$ because $I$ is an ideal. Thus $1_{s(B)} = \sum
1_{W_i} \in I$ as claimed. So the final statement of Lemma~\ref{exact seq} implies that
$I_U \subseteq I$. So $I = I_U$ and hence the map $U\mapsto I_U$ is surjective.
\end{proof}

In the situation where $R=\mathbb{F}$ is a field, all ideals are basic and
the following is immediate.

\begin{cor}\label{cor:field}  Let $G$ be an ample Hausdorff groupoid and let  $\mathbb{F}$ be a  field. Suppose that $G$ is strongly
effective.  Then 
\[
    U \mapsto I_U:=\{f \in A_F(G) : \supp f \subseteq  G_U\}
\]
is an isomorphism from the lattice of open invariant subsets of $\go$ onto the lattice of ideals of $A_{\mathbb{F}}(G)$.
\end{cor}

If in addition the  groupoid  $G$ is second countable and amenable, then
\cite[Corollary~5.9]{BCFS} shows that there is a lattice isomorphism between the open
invariant subsets of $\go$ and the closed ideals of the $C^*$-algebra $C^*(G) = C^*_r(G)$.  Combining
this with Corollary~\ref{cor:field} gives the following.

\begin{cor}
Let $G$ be an ample Hausdorff groupoid. 
 Suppose that $G$ is second countable, amenable and strongly effective.
%\begin{enumerate}
%\item 
If we regard $A_{\C}(G)$ as a $^*$-subalgebra of $C^*(G)$, then the closure
    operation is a lattice isomorphism from the ideals of $A_{\C}(G)$ to the closed
    ideals of $C^*(G)$.
%\item If  $I_1$ and $I_2$ are ideals in $A_{\C}(G)$ with
%$\overline{I_1}=\overline{I_2}$ in $C^*(G)$, then $I_1 = I_2$.
%\end{enumerate}
\end{cor}

%%%%%%%%%%%%%%%%%%%%%%%%%%%%%%%%%%%%%%%%%%%%%%%%%%

\section{Nonbasic ideal structure}\label{sec:nonbasic}

\begin{prop}
\label{prop:ideals} Let $G$ be an ample Hausdorff groupoid and let $R$ be a commutative ring with identity. Suppose that $G$ is  strongly effective, and let
$I$ be an ideal in $A_R(G)$.  Then
\[
    I = \lsp\{r1_B:B \text{ is a compact open bisection and } r1_{s(B)} \in I\}.
\]
\end{prop}

The proof of this proposition uses two technical lemmas. The first is about compact open
bisections in the complement of the unit space of a strongly effective groupoid. The
second shows that restriction of functions to the unit space in such groupoids respects
ideal structure.

\begin{lemma}
\label{lem:getbisections} Let $G$ be an ample Hausdorff groupoid which is strongly effective.
 Suppose that $U \subseteq \go$ is a compact open set and $B \subseteq G
\setminus \go$ is a compact open bisection such that $s(B),r(B) \subseteq U$.  Then there
is a finite collection of compact open bisections $\{N_1, \dots, N_m\}$ such that
\begin{enumerate}
\item\label{it1:lgb} for each $i$, $r(N_i) \subseteq U$;
\item \label{it2:lgb} for each $i$, $N_i^{-1}BN_i = \emptyset$;
\item\label{it3:lgb} $s(N_i) \cap s(N_j) = \emptyset$ for $i \neq j$; and
\item\label{it4:lgb} $s(B)= \bigsqcup_i s(N_i)$.
\end{enumerate}
\end{lemma}
\begin{proof} Since $G$ is strongly effective it is effective, and hence $B \setminus \iso(G)\neq\emptyset$.
For each $\gamma \in B \setminus \iso(G)$, we can apply \cite[Claim~3.2]{BCFS} to obtain
a compact open bisection $V_\gamma \subseteq s(B)$ such that $\gamma \in BV_\gamma$ and
$V_\gamma B V_\gamma = \emptyset$. Let
\[
    C:= \go \setminus \bigcup_{\gamma\in B \setminus \iso(G)} [V_{\gamma}].
\]
Since each $[V_\gamma]$ is an open invariant set, $C$ is closed and invariant. We have \[B
\setminus \iso(G) \subseteq \bigcup_{\gamma \in B \setminus \iso(G)} BV_\gamma \subseteq G
\setminus G_C,\] and hence $B \cap G_C\subseteq \iso(G)$. We claim that $B \cap G_C = \emptyset$.  Suppose, aiming for a contradiction, that $B \cap G_C \neq \emptyset$. Then $B\cap G_C$ is a nonempty open compact bisection of $G_C$. Since $G$ is strongly effective,
$G_C$ is effective. Thus $(B \cap G_C )\setminus \iso(G_C)\neq\emptyset$, contradicting that $B \cap G_C\subseteq \iso(G_C)$. Thus $B \cap G_C = \emptyset$. 
Now
\[
B = B \setminus G_C = \bigcup_{\gamma \in B \setminus \iso(G)} B [V_\gamma],
\]
and in particular, $s(B) \subseteq \bigcup_{\gamma \in B} [V_\gamma]$.

Fix $x \in s(B)$. We will construct a compact open bisection $M_x$ such that $x \in
s(M_x)$, $r(M_x) \subseteq U$, and and $M_x^{-1} B M_x = \emptyset$. Choose $\gamma$ with
$x \in [V_\gamma]$ and choose $\eta \in V_\gamma G x$. First suppose that $r(\eta) \not=
x$. 
Since $G$ is ample and  Hausdorff, there exist compact open neighbourhoods $U_x$ of $x$ and $U_{r(\eta)}$ of $r(\eta)$ such that $U_x\cap U_{r(\eta)}=\emptyset$. We have $\eta\in s^{-1}(U_x)\cap r^{-1}(U_{r(\eta)})$. Thus by intersecting an open compact bisection containing $\eta$ with the closed set $s^{-1}(U_x)\cap r^{-1}(U_{r(\eta)})$ we obtain a compact open bisection $M_x$ containing $\eta$ such that $r(M_x) \cap
s(M_x) = \emptyset$. Since $r(\eta) \in V_\gamma$, we can replace $M_x$ with
$V_\gamma M_x$ to obtain $r(M_x) \subseteq V_\gamma \subseteq U$ (the replacement makes the range smaller, so we still have $r(M_x) \cap s(M_x) = \emptyset$).  We then have
\[
M_x^{-1} B M_x = M_x^{-1} V_\gamma B V_\gamma M_x = \emptyset.
\]
Now suppose that $r(\eta) = x$. Then $x = r(\eta) \in V_\gamma$ and so $M_x := BV_\gamma$
satisfies $x \in s(M_x)$, $r(M_x) \subseteq r(B) \subseteq U$, and $r(M_x) \cap s(M_x) =
\emptyset$. Furthermore,
\[
M_x^{-1} B M_x = V_\gamma B^{-1} B BV_\gamma = V_\gamma s(B)BV_\gamma = V_\gamma B V_\gamma = \emptyset.
\]

We have $s(B) \subseteq \bigcup_{x \in s(B)} s(M_x)$, and since $s(B)$ is compact, there
exist $M_1, \dots, M_n \in \{M_x : x \in s(B)\}$ such that $s(B) \subset \bigcup^n_{i=1}
s(M_i)$. For $1 \le i \le n$ let 
\[
N_i := M_i \setminus \big(\bigcup_{j < i} M_i
s(M_j)\big).
\]
These $N_i$ satisfy (\ref{it1:lgb})--(\ref{it3:lgb}) because the $M_i$ do.
Since each $s(N_i) = s(M_i) \setminus \bigcup_{j < i} s(M_j)$, the $s(N_i)$ are mutually
disjoint and $\bigcup_i s(N_i) = \bigcup_i s(M_i) = s(B)$, giving~(\ref{it4:lgb}).
\end{proof}

\begin{lemma}
\label{lem:mainone} Let $G$ be an ample  Hausdorff groupoid which is strongly effective, and let $R$ be a commutative ring with identity. Let
$I$ be an ideal in $A_R(G)$. If $f \in I$, then $f|_{\go} \in I$.
\end{lemma}
\begin{proof}
Fix $f \in I$. Since $\go$ is closed and open, we can write $f = \sum_{V \in F_0} r_V 1_V +
\sum_{C \in F_1} r_C 1_C$,  where $F_0$  and $F_1$ are finite collections of mutually disjoint compact
open bisections in  $\go$  and  $G \setminus \go$ respectively, and the $r_U$ and $r_V$ are all nonzero in $R$. It suffices to show that  $r_V 1_V \in I$ for all $V\in F_0$. Fix $U
\in F_0$. We have
\[
1_U f 1_U =r_U 1_U + \sum_{B \in F_1} r_B 1_{UBU}  \in I,
\]
{and $\{UBU:B\in F_1\}$} is a set of mutually disjoint compact open bisections contained  in $G \setminus \go$. We will show that $r_U1_U \in I$; we will do this by induction. 

Let $n\geq 0$. Our inductive hypothesis is: if $r 1_U + \sum_{B \in F} r_B 1_B\in I$ where $U \subseteq \go$ is compact open and $F$ is a set of $n$ mutually disjoint compact open bisections in $UGU \setminus \go$, then $r1_U\in I$. When $n=0$ the induction hypothesis holds trivially. 

Now let $g \in I$ be of the form
\[
    g=r1_U  + \sum_{B\in H} r_B1_B
\]
where $U \subseteq \go$ is compact open and $H$ is a collection of $n+1$ mutually
disjoint compact open bisections in $UGU \setminus \go$.

Fix $B_0 \in H$.  We will first show that $a1_{s(B_0)} \in I$.  Since $G$ is strongly effective, we apply
Lemma~\ref{lem:getbisections} to $U$ and $B_0$ to obtain compact open bisections $\{N_1,
\dots, N_m\}$ satisfying properties (\ref{it1:lgb})--(\ref{it4:lgb}) of the lemma.
For $0\leq i \leq m$,  we have $s(N_i)=N_i^{-1}UN_i$ and  $N_i^{-1}B_0N_i=\emptyset$ by properties \eqref{it1:lgb} and \eqref{it2:lgb} of Lemma~\ref{lem:getbisections}, respectively. Hence  
\[
h_i:=1_{N_i^{-1}}g1_{N_i} = r1_{s(N_i)}+ \sum_{B\in H, B\neq B_0} r_B1_{N_i^{-1}BN_i} \in I.
\]
By property (\ref{it4:lgb}) of Lemma~\ref{lem:getbisections} we have
 $\sum_{i=1}^m  r1_{s(N_i)} = r1_{s(B_0)}$, and thus
\begin{equation}\label{eq1}
\sum_{i=1}^m  h_i =r1_{s(B_0)} + \sum_{i=1}^m \sum_{B\in H, B \neq B_0} r_B1_{N_i^{-1}BN_i}
    \in I.
\end{equation}
For $i\neq j$, by property \eqref{it3:lgb} of Lemma~\ref{lem:getbisections} we have 
\begin{gather*}
s(N_i^{-1}BN_i) \cap s(N_j^{-1}BN_j) \subseteq s(N_i)\cap s(N_j)= \emptyset,\\
r(N_i^{-1}BN_i) \cap r(N_j^{-1}BN_j)  \subseteq s(N_i)\cap s(N_j)=\emptyset.
\end{gather*}
Hence  $N_i^{-1}BN_i\cap N_j^{-1}BN_j=\emptyset$. For $B\in H$ with $B\neq B_0$ set
\[D_B := \bigsqcup_{i=1}^n N_i^{-1}BN_i.\]
Then each $D_B$ is a compact open bisection contained in $UGU$ because the source and range of the $N_i$ are contained in $U$ by properties  \eqref{it1:lgb}  and \eqref{it4:lgb} of Lemma~\ref{lem:getbisections}.
Hence \eqref{eq1} is
\begin{equation}\label{apply-ih}
 r1_{s(B_0)}+ \sum_{B\in H\setminus\{B_0\}} r_B1_{D_B} \in I.\end{equation}
To apply the inductive hypothesis, we must verify that each $D_B \cap \go = \emptyset$.
Fix $\gamma \in D_B$.  Then $\gamma \in N_i^{-1}BN_i$ for some $i$.  If $r(\gamma) \neq s(\gamma)$, then $\gamma \notin \go$.  So suppose that
$r(\gamma)=s(\gamma)$. Since $N_i$ is a bisection, there is a unique element $\alpha \in
N_i$ such that $s(\alpha)=s(\gamma)$ and $\gamma = \alpha^{-1} \beta \alpha$ where
$\beta$ is the unique element of $B$ with $s(\beta) = r(\alpha)$.  Since $B \cap \go =
\emptyset$, $\beta \notin \go$, and so $\gamma \notin \go$. Thus $D_B \cap \go = \emptyset$.
Now the inductive hypothesis applies to \eqref{apply-ih}, giving $r1_{s(B_0)} \in I$.

Since our choice of $B_0$ was arbitrary, we obtain $r1_{s(B)} \in I$ for every $B \in H$.
We may also assume the collection
$\{s(B)\}_{B \in H}$ is disjoint (by disjointification). So
\[
V:=\bigcup_{B\in H}s(B) \subseteq U
\]
satisfies
\[
    r1_V = \sum_{B \in H} r1_{s(B)} \in I.
\]
Since $s(B) \cap U\setminus V = \emptyset$ for $B \in H$ we have
 $1_{U\setminus V} g 1_{U \setminus V} = r1_{U \setminus V} \in I$.  Thus  $r1_U =
r1_V + r1_{U \setminus V} \in I$ as well.
\end{proof}

\begin{proof}[Proof of Proposition~\ref{prop:ideals}:]
It suffices to show that \[I \subseteq \lsp\{r1_B: \text{$B$ is a compact open bisection and } r1_{s(B)} \in I\}.\] Fix $f \in I$. Since $G$ is strongly effective, 
Lemma~\ref{lem:mainone} implies that $f|_{\go} \in I$, and hence $f - f|_{\go} \in I$. Since $f$ is locally constant,  $A := f(\go) \subseteq R$ is finite. The sets $B_r := f^{-1}(r)\cap \go$, $r \in A$ are
mutually disjoint compact open subsets of $\go$, and we have $f|_{\go} = \sum_{r \in A} r
1_{B_r}$. Each $r 1_{B_r} = 1_{B_r} f|_{\go} \in I$ and we deduce that $f|_{\go} \in
\lsp\{r1_B: r1_{s(B)} \in I\}$.

So it suffices to show that $g := f - f|_{\go} \in \lsp\{r1_B: r1_{s(B)} \in I\}$.
Express $g = \sum_{B \in F} r_B 1_B$ where $F$ is a finite set of mutually disjoint
compact open bisections in $G \setminus \go$. Fix $C \in F$; we just have to
establish that $r_C 1_{s(C)} \in I$. We have
\[
1_{C^{-1}}g = r_{C}1_{s(C)} + \sum_{B \in F \setminus\{C\}} r_{B} 1_{C^{-1}B} \in I.
\]
We claim that for each $B \in F \setminus \{C\}$ we have $C^{-1}B \subseteq G \setminus
\go$. Fix $B \in F \setminus \{C\}$ and $\gamma \in C^{-1}B$. Then
$\gamma=\alpha^{-1}\beta$ for some $\alpha \in C$ and $\beta \in B$.  Since $C \cap B
= \emptyset$, $\alpha \neq \beta$, and so $\gamma \notin \go$. Thus $(1_{C^{-1}}g)|_{\go}=r_C 1_{s(C)}$.  Since $G$ is strongly effective,  Lemma~\ref{lem:mainone} gives $r_{C}1_{s(C)} \in I$ as needed.
\end{proof}

 For any ring $R$, we write 
$
\mathcal{L}(R) := \{I : I\text{ is a two-sided ideal of }R\}
$ for the set of ideals of $R$.  We now state our main theorem.

\begin{thm}
\label{thm:bijection} Let $G$ be an ample Hausdorff groupoid which is strongly effective, and let  R be a commutative ring with identity.  Let $\mathcal{O}$ be the set of all nonempty open invariant subsets of
$\go$, and let $\mathcal{F}$ be the set of all functions $\pi:\mathcal{O} \to
\mathcal{L}(R)$ such that for all $\mathcal{A} \subseteq \mathcal{O}$
\begin{equation}\label{eq:2cain}\textstyle
    \pi(\bigcup_{U \in \mathcal{A}}U) = \bigcap_{U \in \mathcal{A}}\pi(U).
    \end{equation}
There is a bijection $\Gamma:\mathcal{F} \to \mathcal{L}(A_R(G))$ such that
\[
\Gamma(\pi) = \lsp_R\Big(\bigcup_{U \in \mathcal{O}} \{rf : r \in \pi(U), f \in A_R(G), \supp(f) \subseteq G_U\}\Big).
\]
For each $U \in \mathcal{O}$, we have
\[
\pi(U) = \{r \in R : r1_B \in \Gamma(\pi)\text{ for all compact open }B \subseteq U\}.
\]
\end{thm}

The following observation will be useful a couple of times: suppose that $V,W \in
\mathcal{O}$ with $V \subseteq W$. Then $\mathcal{A} := \{V, W\}$ satisfies $\bigcup_{U
\in \mathcal{A}} U = W$. So if $\pi$ satisfies~\eqref{eq:2cain}, we have $\pi(V) \cap
\pi(W) = \pi(W)$. Hence~\eqref{eq:2cain} implies that
\begin{equation}\label{eq:1cain}
    V \subseteq W \Longrightarrow \pi(W) \subseteq \pi(V)\quad\text{ for all $V,W \in \mathcal{O}$.}
\end{equation}

Before proving the theorem, we establish a lemma.

\begin{lemma}\label{lem:containment}
Resume the notation of Theorem~\ref{thm:bijection}. Let $\pi_1, \pi_2 \in \mathcal{F}$. Then 
\[
\Gamma(\pi_1) \subseteq \Gamma(\pi_2) \ \text{ if and only if }\ 
\pi_1(U) \subseteq \pi_2(U)\text{ for all }U \in \mathcal{O}.
\]
\end{lemma}
\begin{proof}  Suppose that $\Gamma(\pi_1)\subseteq \Gamma(\pi_2)$. Fix $U \in \mathcal{O}$.
Fix  $r \in \pi_1(U)$. Let $u \in U$ and  let $K$ be a compact open
subset of $U$ with $u\in K$.   Then  $r1_K \in  \Gamma(\pi_1) \subseteq \Gamma(\pi_2)$.
By definition of $\Gamma(\pi_2)$, 
\[
    r1_K = \sum_{V \in \mathcal{O}}r_V f_V,
\]
where $r_V \in \pi_2(V)$ and $\supp(f_V) \subseteq G_V\cap K=V\cap K$.  Let $F$ be the finite collection of $V
\in \mathcal{O}$ such that $f_V(u) \neq 0$. Then $r=\sum_{V \in F} r_Vf_V(u)$. Let
\[
    W_u := \bigcap_{V \in F} V\cap U.
\]
For $V \in F$, we have $W_u \subseteq V$, and so~\eqref{eq:1cain} gives $r_V\in \pi_2(V) \subseteq
\pi_2(W_u)$. Since $\pi_2(W_u)$ is an ideal, it follows that $r=\sum_{V \in F} r_Vf_V(u) \in \pi_2(W_u)$. 
 Now~\eqref{eq:2cain} gives
\[
    \pi_2(U) = \pi_2\big(\bigcup_{u\in U}W_u\big)=\bigcap_{u \in U} \pi_2(W_u).
\]
Thus $r \in \pi_2(U)$, and hence $\pi_1(U) \subseteq \pi_2(U)$.

If  $\pi_1(U) \subseteq \pi_2(U)$ for all $U \in \mathcal{O}$, then it is immediate from the definition of $\Gamma$ that $\Gamma(\pi_1) \subseteq \Gamma(\pi_2)$.
\end{proof}

\begin{proof}[Proof of Theorem~\ref{thm:bijection}]
Since each $\pi(U)$ is an ideal of $R$ and each $I_U$ is an ideal of $A_R(G)$, it follows that $\Gamma(\pi)$ is an ideal in $A_R(G)$.

To see that $\Gamma$ is injective, suppose that $\Gamma(\pi_1)=\Gamma(\pi_2)$. Then two
applications of Lemma~\ref{lem:containment} show that $\pi_1(U) = \pi_2(U)$ for every
$U$, and so $\pi_1 = \pi_2$. Hence $\Gamma$ is injective.

To see that $\Gamma$ is surjective,  let $I$ be an ideal in $A_R(G)$. Let $U\in\mathcal{O}$. Set
\begin{equation}\label{eqn:piu}
    \pi(U) = \{r\in R : r1_B \in I \text{ for all compact open } B\subseteq U\}.
\end{equation}
Then  $\pi(U) \in\mathcal{L}(R)$, and we claim that $\pi\in \mathcal{F}$, that is, $\pi$
satisfies~\eqref{eq:2cain}. 

Let  $\mathcal{A} \subseteq \mathcal{O}$. Since $\pi$ reverses set inclusion,  we have 
$\pi(\bigcup_{U \in \mathcal{A}}U)\subseteq \pi(U)$ for all $U\in\mathcal{A}$. Thus  $\pi(\bigcup_{U \in \mathcal{A}}U)\subseteq \bigcap_{U \in \mathcal{A}}\pi(U)$. 

For the reverse containment, fix $r \in \bigcap_{U \in
\mathcal{A}}\pi(U)$.  Let $B$ be a compact open subset of $\bigcup_{U \in \mathcal{A}}U$. We need to show $r1_B\in I$. For each $b \in B$, there exists $U_b
\in \mathcal{A}$ such that $b \in U_b$.  Because $\go$ has a basis of compact open sets,
there exists compact open $K_b \subseteq U_b$ such that $b \in K_b$. Since $r\in\pi(U_b)$ we have $r1_{K_b} \in I$.  Since $B$ is compact, there is
a finite set $C \subseteq \{K_b : b\in B\}$ that covers $B$.  Since $I$ is an ideal,
$r1_{K_b} \in I$ implies $r1_K \in I$ for any compact open $K \subseteq K_b$.  So we may
disjointify $C$ to obtain a finite cover, still called $C$, of $B$ by compact open sets
satisfying $r1_K \in I$ for all $K \in C$.
So  \[r1_B = \sum_{K \in C}r1_K \in I.\]  Thus $r \in \pi(\bigcup_{U \in \mathcal{A}}U)$. Thus 
$\pi$ satisfies~\eqref{eq:2cain}, and $\pi\in\mathcal{F}$ as claimed.

Finally, we show that $\Gamma(\pi) = I$. 
To see that $\Gamma(\pi) \subseteq I$, we take $U \in \mathcal{O}$, $r \in \pi(U)$ and $f \in A_R(G)$ with $\supp(f)\subseteq G_U$, so that $rf$ is a typical spanning element of $\Gamma(\pi)$.  Write $rf=\sum_{B\in F}r_B1_B$ where $F$ is a finite set of mutually disjoint compact open bisections and $0\neq r_B$ for $B\in F$.  Fix $L\in F$ and take $\gamma\in L$.  Then $(rf)(\gamma)=r_{L}\in \pi(U)\setminus \{0\}$.   Since $\supp(rf)\subseteq\supp f\subseteq G_U$ we must have $s(L)\subseteq U$ and hence $r_L1_{s(L)}\in I$ by definition of $r_L\in\pi(U)$. Thus $r_L1_L=1_L*(r_L1_{s(L)})\in I$.  Thus $rf\in I$, and hence $\Gamma(\pi) \subseteq I$.

Conversely, fix $f \in I$. Because $G$ is strongly effective, by Proposition~\ref{prop:ideals} we have $f= \sum_{B\in F}r_B1_B$ where each $r_B1_{s(B)} \in I$. Fix $L \in F$. Recall that $[s(L)]$ is the smallest invariant subset of $G^{(0)}$ containing $s(L)$.
We claim that $r_{L}1_K \in I$ for every compact open $K\subseteq [s(L)]$; this implies $r_{L} \in \pi([s(L)])$ and hence $r_{L}1_{L} \in \pi([s(L)])I_{s([L])} \subseteq \Gamma(\pi)$. It then follows that $f\in \Gamma(\pi)$.

 To prove the claim, fix
$K\subseteq [s(L)]$.  For each $k \in K$, there exists $\gamma_k$ such that $s(\gamma_k)
\in s(L)$ and $r(\gamma_k) = k$. Let $B_k$ be a compact open bisection containing
$\gamma_k$.  We can assume that $s(B_k) \subseteq s(L)$ and $r(B_k) \subseteq K$ (by
taking intersections). Now $\{r(B_k):k\in K\}$ is an open cover of $K$.
By taking a finite subcover and disjointifying, we get a collection of compact open
bisections $\{B_1, \dots, B_n\}$ whose ranges form a disjoint cover of $K$.
For $1\leq i\leq n$ we have $r_L1_{B_i} = 1_{B_i}*r_L1_{s(L)} \in I$ and hence
$r_L1_{r(B_i)} = r_L1_{B_i}*1_{B_i^{-1}} \in I$.  Now
$r_L1_K = \sum_{i=1}^{n} r_L1_{r(B_i)} \in I$
as claimed. Thus $I\subseteq \Gamma(\pi)$. Now $I=\Gamma(\pi)$ and we have shown that $\Gamma$ is surjective.  By definition of
$\pi$---see~\eqref{eqn:piu}---this also establishes the final statement of the
theorem. 
\end{proof}

\section{The lattice isomorphism}
\label{sec:lattice} In this section, we study the  lattice structure of the set $\mathcal{L}(A_R(G))$ of
ideals of $A_R(G)$. We have established in Theorem~\ref{thm:bijection} a bijection from
\[\mathcal{F}:= \{\pi: \mathcal{O} \to \mathcal{L}(R) : \pi \text{
satisfies~\eqref{eq:2cain}}\}\] onto $\mathcal{L}(A_R(G))$. Since $(\mathcal{L}(A_R(G)),
\subseteq, +,\cap)$ is a lattice, $\Gamma$ induces a lattice structure $(\mathcal{F},\preccurlyeq, \vee,\wedge)$  on $\mathcal{F}$ via 
\[
\pi_1 \preccurlyeq \pi_2\Longleftrightarrow \Gamma(\pi_1) \subseteq
\Gamma(\pi_2).
\]
However, it seems  difficult to explicitly describe the element $\pi_1 \vee
\pi_2 \in \mathcal{F}$ such that $\Gamma(\pi_1 \vee \pi_2) = \Gamma(\pi_1) +
\Gamma(\pi_2)$.  

Here we start by explaining the difficulties with $(\mathcal{F},\preccurlyeq, \vee,\wedge)$, and then present a new parameterisation $\mathcal{F}'$ of the ideals of $A_R(G)$ which is
better suited to describing the lattice structure. We will also see that $\mathcal{F}'$ has the additional advantage that it does not require a computation of the lattice $\mathcal{O}$ of open invariant subsets of $\go$.

Let $\pi_1,\pi_2\in\mathcal{F}$. By Lemma~\ref{lem:containment} we have 
\[
\pi_1 \preccurlyeq \pi_2 \Longleftrightarrow \pi_1(U) \subseteq \pi_2(U) \text{ for all } U \in \mathcal{O}.
\]
It is then easy to verify that the function $U \mapsto \pi_1(U) \cap \pi_2(U)$ belongs to $\mathcal{F}$, and is the meet $\pi_1\wedge\pi_2$ of $\pi_1$ and $\pi_2$. The join $\pi_1\vee\pi_2$ in $\mathcal{F}$ is more complicated. One might
guess that $\pi_1 \vee \pi_2$ is the function $g$ defined by  $g(U)= \pi_1(U) + \pi_2(U)$ for  $U\in\mathcal{O}$. But the next example shows that  $g$ may not even  belong to $\mathcal{F}$.

\begin{ex}
\label{ex:twopoints} Consider the groupoid $G$ that consists of two units $x$ and $y$ with the discrete
topology. That is, $G = \go = \{x,y\}$.  Then $A_{\mathbb{Z}}(G) = \mathbb{Z} \oplus
\mathbb{Z}$. The set of nonempty open invariant subsets of $\go$ is
\[
    \mathcal{O}= \{\{x\}, \{y\}, \go\}.
\]
Define $\pi_1, \pi_2: \mathcal{O} \to \mathcal{L}(\mathbb{Z})$ by
\begin{align*}
\pi_1(\{x\}) &= 2\mathbb{Z}, \ \pi_1(\{y\}) = 3\mathbb{Z}, \ \pi_1(\go) = 6\mathbb{Z}, \\
\pi_2(\{x\}) &= 3\mathbb{Z},\ \pi_2(\{y\}) = 5\mathbb{Z}, \ \pi_2(\go) = 15\mathbb{Z}.
\end{align*}
Then $\pi_1$ and $\pi_2$ satisfy~\eqref{eq:2cain}. Also $\Gamma(\pi_1) = 2\mathbb{Z}
\oplus 3\mathbb{Z}$ and $\Gamma(\pi_2) = 3\mathbb{Z} \oplus 5\mathbb{Z}$. Hence
$\Gamma(\pi_1) + \Gamma(\pi_2) = (2 \mathbb{Z} \oplus 3 \mathbb{Z}) + (3 \mathbb{Z}
\oplus 5 \mathbb{Z}) = \mathbb{Z} \oplus \mathbb{Z}$, and it follows that $\pi_1 \vee
\pi_2$ is given by
\[
(\pi_1 \vee \pi_2)(\{x\}) = (\pi_1 \vee \pi_2)(\{y\}) = (\pi_1 \vee \pi_2)(\{x,y\}) = \mathbb{Z}.
\]
Since $\pi_1(\go) + \pi_2(\go) = 6\mathbb{Z} + 15\mathbb{Z} = 3\mathbb{Z} \not=
\mathbb{Z}$, we see that $\pi_1 \vee \pi_2$ is not given by pointwise addition of ideals.
Indeed, since $\go = \{x\} \cup \{y\}$ but $\pi_1(\go) + \pi_2(\go) = 3\mathbb{Z} \not=
\mathbb{Z} = \big(\pi_1(\{x\}) + \pi_2(\{x\})\big) \cap \big(\pi_1(\{x\}) +
\pi_2(\{x\})\big)$, we see that $U \mapsto \pi_1(U) + \pi_2(U)$ does not
satisfy~\eqref{eq:2cain}.
\end{ex}

To overcome this problem, we will reparameterise  $\mathcal{F}$ in
terms of a set $\mathcal{F'}$ of functions from $G^{(0)}$ to $\mathcal{L}(R)$ that are
continuous with respect to a suitable topology, and are suitably $G$-invariant. We will show in Theorem~\ref{thm:alt description} below that the order relation and the meet and join
operations on $(\mathcal{F}', \preccurlyeq)$ translate to pointwise containment,
intersection and addition of functions, giving  a natural description of the
lattice structure on $\mathcal{L}(A_R(G))$.

Define a topology on $\mathcal{L}(R)$ as follows: Given a finite set $F \subseteq R$,
define
\[
    Z(F):= \{I \in \mathcal{L}(R) : F \subseteq I\}.
\]
Then $Z(F_1) \cap Z(F_2) = Z(F_1 \cup F_2)$ for finite $F_1, F_2 \subseteq R$, and
hence the collection of all such $Z(F)$ forms a basis for a topology on
$\mathcal{L}(R)$. We equip $\mathcal{L}(R)$ with this topology. It is a fairly weak topology: it is  $T_0$ because if $I,J \in
\mathcal{L}(R)$ and $r \in I \setminus J$, then $Z(\{r\})$ contains $I$ but not $J$.
However it is not a $T_1$ topology: if $J \subseteq I$ then every open set containing $J$
contains $I$.  

The following lemma is straightforward to prove.
\begin{lemma}\label{lem-rho-cty}
Let $\rho:\go \to \mathcal{L}(R)$. Then $\rho$ is continuous at $u \in \go$ if and only if for all $a \in
\rho(u)$ there exists an open neighbourhood $W$ of $u$ such that $a \in \rho(w)$ for
every $w \in W$. 
\end{lemma}
\begin{proof} Fix $u\in\go$. The sets $\{Z(\{a\}):a\in\rho(u)\}$ form a neighbourhood subbasis at $\rho(u)$ for the topology on $\mathcal{L}(R)$.  Thus  $\rho$ is continuous at $u$ if and only if for each $a \in
\rho(u)$ there is an open neighbourhood $W$ of $u$ such that $\rho(W)\subseteq Z(\{a\})$. 
\end{proof}

We say $\rho:\go \to \mathcal{L}(R)$ is \emph{$G$-invariant} if $\rho(s(\gamma))=\rho(r(\gamma))$
for all $\gamma \in G$. We set 
\[
\mathcal{F}' := \{\rho:\go \to \mathcal{L}(R) : \rho \text{ is $G$-invariant and continuous}\}.
\] 

\begin{lemma}
\label{lem:rhopi} Let $G$ be an ample Hausdorff groupoid and let $R$ a commutative ring with identity.
\begin{enumerate}
\item\label{it:rhopi1} For any function $\rho : \go \to \mathcal{L}(R)$, the function
    $\pi_\rho : \mathcal{O} \to \mathcal{L}(R)$ given by
    \begin{equation}\label{eq:pirho}
        \pi_\rho(U) = \bigcap_{u \in U} \rho(u)
    \end{equation}
    satisfies~\eqref{eq:2cain}.
\item\label{it:rhopi2} For any function $\pi : \mathcal{O} \to \mathcal{L}(R)$, the
    formula
    \begin{equation}\label{eq:rhopi}
        \rho_\pi(u) = \bigcup_{U\text{open}, u \in U} \pi([U])
    \end{equation}
    defines a $g$-invariant continuous function $\rho_\pi : \go \to \mathcal{L}(R)$.
\item\label{it:rhopi3} We have $\pi_{\rho_\pi} = \pi$ for $\pi \in \mathcal{F}$ and
    $\rho_{\pi_\rho} = \rho$ for $\rho \in \mathcal{F'}$. In particular, $\rho\mapsto\pi_\rho$ is a bijection from $\mathcal{F}'$ to $\mathcal{F}$.
\end{enumerate}
\end{lemma}
\begin{proof}
(\ref{it:rhopi1}) Given $\rho : \go \to \mathcal{L}(R)$ and $\mathcal{A} \subseteq
\mathcal{O}$, we have
\[
\pi_\rho\Big(\bigcup_{U \in \mathcal{A}} U\Big)
    = \bigcap_{u \in \bigcup_{U \in \mathcal{A}} U} \rho(u)
    = \bigcap_{U \in \mathcal{A}} \Big(\bigcap_{u \in U} \rho(u)\Big)
    = \bigcap_{U \in \mathcal{A}} \pi_\rho(U).
\]

(\ref{it:rhopi2}) Fix $\pi \in \mathcal{F}$. We start by showing that $\rho_\pi(u) \in \mathcal{L}(R)$. Fix $u \in \go$. Let $r \in \rho_\pi(u)$ and $s \in R$. By definition of $\rho_{\pi}$ there exists an 
open neighbourhood $U$ of $u$ such that $r \in \pi([U])$. Then $rs, sr \in \pi([U])
\subseteq \rho_\pi(u)$ because $\pi([U])$ is an ideal. Also, if $r,s \in \rho_\pi(u)$, there exist open neighbourhoods $U_r$ and $U_s$ of $u$ such that $r \in \pi([U_r])$ and $s
\in \pi([U_s])$. Now~\eqref{eq:1cain} gives $r,s \in \pi([U_r \cap U_s])$, and hence $r +
s \in \pi([U_r \cap U_s]) \subseteq \rho_\pi(u)$. Thus $\rho_\pi(u)$ is an ideal. 

To see that $\rho_\pi$ is continuous, we use Lemma~\ref{lem-rho-cty}. Fix  $u\in\go$ and fix $a \in \rho_{\pi}(u)$. By definition of $\rho_\pi$, there exists an open neighbourhood $W$ of $u$ such that $a \in \pi([W])$.  Let $v\in W$. Then
\[
a\in\pi([W])\subseteq \bigcup_{V\text{\ open}, v\in V}\pi([V])=\rho_\pi(v).
\]
It follows that $\rho_\pi$ is continuous.  It is $G$-equivariant because $\{[U] : r(\gamma) \in U, U\text{ open}\} = \{[U] : s(\gamma) \in U, U\text{ open}\}$ for every $\gamma \in G$.

(\ref{it:rhopi3}) First fix $\pi \in \mathcal{F}$. We must
show that $\pi_{\rho_\pi}(U) = \pi(U)$ for all $U \in \mathcal{O}$. Fix  $U \in \mathcal{O}$. First suppose that $a \in \pi(U)$. Then $a \in
\pi([W])$ for every open $W \subseteq U$. Since $[W]\subseteq U$ implies $\pi(U)\subseteq \pi([W])$, we get
\[
a \in \bigcap_{u \in U} \Big(\bigcup_{W \text{ open},\; u \in W} \pi([W])\Big)
    = \bigcap_{u\in U} \rho_{\pi}(u)
    = \pi_{\rho_{\pi}}(U).
\]
Now suppose that $a \in \pi_{\rho_{\pi}}(U)$. Then for each $u \in U$, there exists an
open neighbourhood $W_u \subseteq U$ of $u$ such that $a \in \pi([W_u])$. Since $\pi$
satisfies~\eqref{eq:2cain}, we obtain
\[
    a \in \bigcap_{u \in U} \pi([W_u]) = \pi\Big(\bigcup_{u \in U} [W_u]\Big) = \pi(U).
\]

Now fix $\rho \in \mathcal{F}'$. We must show that $\rho_{\pi_{\rho}}= \rho$. Fix $u \in
\go$. Using that $\rho$ is $G$-invariant for the final equality, we calculate:
\begin{align}
\rho_{\pi_{\rho}}(u)
    &= \bigcup_{W \text{ open},\; u \in W} \pi_{\rho}([W])\notag\\
    &=  \bigcup_{W \text{ open},\; u \in W} \Big(\bigcap_{v \in [W]}\rho(v)\Big)\notag\\
    &= \bigcup_{W \text{ open},\; u \in W} \Big(\bigcap_{v \in W}\rho(v)\Big).\label{eq:rhopirho}
\end{align}
To see that this is equal to $\rho(u)$, first fix $a \in \rho_{\pi_{\rho}}(u)$. Then
there is a neighbourhood $U$ of $u$ such that $a \in \rho(v)$ for every $v \in [U]$. In
particular, $a \in \rho(u)$. Now fix $a \in \rho(u)$. Then there exists an open
neighbourhood $W\subseteq \go$ of $u$ such that $\rho(v) \in Z(\{a\})$ for all $v \in W$.
That is $a \in \rho(v)$ for all $v \in W$. So~\eqref{eq:rhopirho} gives
\[
    a \in \bigcap_{v \in W}\rho(v) \subseteq \rho_{\pi_{\rho}}(u).\qedhere
\]
\end{proof}

\begin{thm}\label{thm:alt description}
Let $G$ be an ample Hausdorff groupoid  which is strongly effective, and let $R$ be a commutative ring with identity. Let
$\mathcal{F}'$ be the set of continuous $G$-invariant functions $\rho : \go \to
\mathcal{L}(R)$. There is a bijection $\Gamma':\mathcal{F}' \to \mathcal{L}(A_R(G))$ such
that
\[
    \Gamma'(\rho) = \lsp_R\Big\{r1_B : B\text{ is a compact open bisection and }r \in \bigcap_{u \in [s(B)]} \rho(u)\Big\}.
\]
Define a relation $\preccurlyeq$ on $\mathcal{F}'$ by
\[
    \rho_1 \preccurlyeq \rho_2 \text{ if and only if } \rho_1(u) \subseteq \rho_2(u) \text{ for all }u \in \go.
\]
Then $(\mathcal{F}', \preccurlyeq)$ is a lattice with join and meet operations given by
\begin{align}
\rho_1 \vee \rho_2(u) &= \rho_1(u) + \rho_2(u) \text{ and } \label{eq:vee'}\\
\rho_1 \wedge \rho_2(u) &= \rho_1(u) \cap \rho_2(u), \label{eq:wedge'}
\end{align}
and $\Gamma' : (\mathcal{F}', \preccurlyeq) \to (\mathcal{L}(A_R(G)), \subseteq)$ is a
lattice isomorphism.
\end{thm}

\begin{proof}
To see that $\Gamma'(\rho) = \Gamma(\pi_\rho)$, we start by unravelling $\Gamma(\pi_\rho)$:
\[
\Gamma(\pi_\rho)=\lsp_R\Big(\bigcup_{U \in \mathcal{O}} \{rf : r \in \pi(U), f \in A_R(G), \supp(f) \subseteq G_U\}\Big).
\] 
Take $rf\in \Gamma(\pi_\rho)$.  Then $rf=\sum_{B\in F} r_B1_B$ where $F$ is a set of mutually disjoint compact open bisections contained in $G_U$ for some $U \in \mathcal{O}$. Fix $L\in F$ and let  $\gamma\in L$. Then $r_L=rf(\gamma)\in \bigcap_{u\in U}\rho(u)\subseteq \bigcap_{u\in [s(L)]}\rho(u)$ since $[s(L)]\subseteq U$. Thus each $r_B1_B\in\Gamma'(\rho)$, and hence $rf\in  \Gamma'(\rho)$. Now $\Gamma(\pi_\rho)\subseteq  \Gamma'(\rho)$. The reverse set inclusion is immediate. Thus $\Gamma'(\rho) = \Gamma(\pi_\rho)$.

Now $\Gamma'$ is the composition of the bijections $\rho \mapsto \pi_\rho$ and $\pi\mapsto \Gamma(\pi)$ of 
Lemma~\ref{lem:rhopi} and Theorem~\ref{thm:bijection}, respectively.   Hence $\Gamma'$ is a bijection. To see that it is a lattice isomorphism, we must show that $\Gamma'(\rho_1) \subseteq \Gamma'(\rho_2)$ if and only if
$\rho_1 \preccurlyeq \rho_2$.

First suppose that $\Gamma'(\rho_1) \subseteq \Gamma'(\rho_2)$. Then $\Gamma(\pi_{\rho_1})
\subseteq \Gamma(\pi_{\rho_2})$, forcing $\pi_{\rho_1}(U) \subseteq \pi_{\rho_2}(U)$ for
all $U$. Fix $u \in \go$ and $a \in \rho_1(u)$.  We show that $a \in \rho_2(u)$. We have
\begin{align*}
a \in \rho_1(u) = \rho_{\pi_{\rho_1}}(u) = \bigcup_{W \text{ open},\; u \in W} \pi_{\rho_1}([W]).
\end{align*}
Hence there is an open neighbourhood $W \subseteq \go$ of $u$ such that $a \in
\pi_{\rho_1}([W])$. Let $K$ be a compact open subset of $W$. Then
\[
    a1_K \in \Gamma(\pi_{\rho_1}) = \Gamma'(\rho_1) \subseteq \Gamma'(\rho_2),
\]
forcing
\[
    a \in \bigcap_{v \in [K]} \rho_2(v),
\]
and in particular $a \in \rho_2(u)$. Thus $\rho_1 \preccurlyeq \rho_2$.

Second, suppose that $\rho_1 \preccurlyeq \rho_2$. Then $\rho_1(u) \subseteq \rho_2(u) \text{ for all }u \in \go$, and take $\Gamma'(\rho_1) \subseteq \Gamma'(\rho_2)$ by definition of $\Gamma'$.

It remains only to show that $\rho_1 \vee \rho_2$ and $\rho_1 \wedge \rho_2$ are given by
the formulas \eqref{eq:vee'}~and~\eqref{eq:wedge'}.
For this, define $\tau_\vee,\tau_\wedge : \go \to \mathcal{L}(R)$ by
\[
\tau_\vee(u) = \rho_1(u) + \rho_2(u)\quad\text{ and }\quad
\tau_\wedge(u) = \rho_1(u) \cap \rho_2(u)\quad\text{ for all $u \in \go$.}
\]
We first check that $\tau_\vee  \in \mathcal{F}'$.
To see that $\tau_\vee$ is continuous, we use Lemma~\ref{lem-rho-cty}.
Fix $u \in \go$ and $a \in \tau_\vee(u)$. Write $a = a_1 + a_2$ where $a_1 \in \rho_1(u)$
and $a_2 \in \rho_(u)$. Since $\rho_1$ and $\rho_2$ are continuous, there exist open
neighbourhoods $W_1$ and $W_2$ of $u$ such that $a_i \in \bigcap_{v \in W_i}
\rho_i(v)$ for $i=1,2$. Hence $W := W_1 \cap W_2$ is an open neighbourhood of $v$
such that
\[
    a = a_1 + a_2 \in  \tau_\vee(v) \text{ for all } v \in W.
\]
It follows that $\tau_\vee$ is continuous. It is $G$-equivariant because $\rho_1$ and $\rho_2$ are. Thus $\tau_\vee\in \mathcal{F}'$. A similar argument shows that $\tau_\wedge \in \mathcal{F}'$ as well.

We have $\tau_\wedge(u) = \rho_1(u) \cap \rho_2(u) \subseteq \rho_1(u), \rho_2(u)$ for
all $u\in\go$, and so $\tau_\wedge \preccurlyeq \rho_1, \rho_2$. The maximality of $\rho_1 \wedge
\rho_2$ gives $\tau_\wedge \preccurlyeq \rho_1 \wedge \rho_2$. On the other hand, we have
$\rho_1 \wedge \rho_2 \preccurlyeq \rho_1, \rho_2$, so for all $u\in\go$ we have $(\rho_1
\wedge \rho_2)(u) \subseteq \rho_1(u), \rho_2(u)$, forcing $(\rho_1 \wedge \rho_2)(u)
\subseteq \rho_1(u) \cap \rho_2(u) = \tau_\wedge(u)$, and so $\rho_1 \wedge \rho_2
\preccurlyeq \tau_\wedge$. Since $\preccurlyeq$ is a partial order, we deduce that
$\tau_\wedge = \rho_1 \wedge \rho_2$. A similar argument gives $\tau_\vee = \rho_1
\vee \rho_2$.
\end{proof}

\begin{rmk}
We chose to first present the description of $\mathcal{L}(A_R(G))$ in terms of 
\[\mathcal{F}:= \{\pi: \mathcal{O} \to \mathcal{L}(R) : \pi \text{
satisfies~\eqref{eq:2cain}}\}\] 
of Theorem~\ref{thm:bijection} rather than the description in terms of
\[
\mathcal{F}' := \{\rho:\go \to \mathcal{L}(R) : \rho \text{ is $G$-invariant and continuous}\}
\] 
of Theorem~\ref{thm:alt description}. We did this  because  $\mathcal{F}$ is closer in spirit to the  description, in terms of open invariant sets,  of the ideals in a groupoid $C^*$-algebra in \cite[Corollary 4.9]{Renault:JOT91} or the basic ideals of a Steinberg algebra in Section~\ref{sec:basic}. In the context of graph groupoids, $\mathcal{F}$ is also much more closely
related to the Bates--Pask--Raeburn--Szyma\'nski catalogue of ideals of $C^*(E)$ in \cite[\S4]{BPRS},
Tomforde's catalogue of basic ideals of $L_R(E)$ in terms of saturated hereditary sets in \cite[Theorem~7.9]{Tomforde:JPAA11}, and the analogous theorems for the Kumjian--Pask algebras of higher-rank graphs \cite[Theorem~5.1]{ACaHR} and \cite[Theorem~9.4]{CFaH}.
(see also Theorems~\ref{thm:LPideals} and ~\ref{thm:KPideals} below). Nevertheless, we believe that the description in terms
of $\mathcal{F}'$ is very natural, and at least in some cases much easier to compute with
(as is the case in Example~\ref{ex:tree} below). The main advantage of the description in terms of
$\mathcal{F}$ is that it is easy to decide which elements $r 1_B$ belong to an ideal of
the form $\Gamma(\pi)$; the principal advantages of the description in terms  of
$\mathcal{F}'$ are that there is no need to compute the collection of all open
invariant sets, and that it makes the join operation easier to compute.
\end{rmk}

\section{Leavitt path algebras and Kumjian--Pask algebras}\label{sec:Lpa/KPa}

In this section we explain what 
Theorem~\ref{thm:bijection}, and its crucial ingredient Proposition~\ref{prop:ideals}, say about a Leavitt path algebra of a directed graph and about a Kumjian--Pask algebra of a higher-rank graph.  Since a Leavitt path algebra is a  Kumjian--Pask
algebra of a $1$-graph, we will deduce Theorem~\ref{thm:LPideals} about the Leavitt path algebra from the analogous theorem about the Kumjian--Pask algebra.  We start by gathering background needed to state Theorem~\ref{thm:LPideals}.

Let  $E=(E^0, E^1, r, s)$ be  a row-finite directed graph with no sources.
A subset $H \subseteq E^0$ is \emph{hereditary} if $r(e) \in H$ implies $s(e) \in H$ for all $e \in E^1$, and is \emph{saturated} if $s(vE^1) \subseteq H$ implies $v \in H$ for all $v \in E^0$. 
We write $\mathcal{H}_E$ for the
collection of all saturated hereditary subsets of $E^0$.
Given $\mathcal{A} \subseteq \mathcal{H}_E$, we write $\bigvee \mathcal{A}$ for the
smallest saturated hereditary set containing every $H \in \mathcal{A}$; that is,
\[
\bigvee \mathcal{A} = \bigcap_{H \in \mathcal{H}_E,\; K \subseteq H\text{ for all }K \in \mathcal{A}} H.
\]

A graph $E$ is satisfies \emph{Condition~(L)} if every cycle has an entry. Further, $E$ satisfies \emph{Condition~(K)} if for every $v\in E^0$, either there is no cycle based at $v$, or there are at least two distinct return paths based at $v$.  A graph satisfies  Condition~(K) if and only if  for every saturated hereditary subset $H\neq E^0$ of $E^0$, the subgraph $E\setminus H=(E^0\setminus H, s^{-1}(E^0\setminus H), r, s)$ satisfies Condition~(L) \cite[Lemma~4.7]{CBMS}.  
It  follows from Corollary~\ref{cor-se-iff-sa} below that $E$ satisfies Condition~(K)  if and only if  the graph groupoid of $E$ is strongly effective.
 
 We refer to \cite[\S2]{Tomforde:JPAA11} for the definition of the Leavitt path algebra $L_R(E)$. We write $(p,s)$ for the universal generating Leavitt family in  $L_R(E)$. 
Let $H\in \mathcal{H}_E$. Then the ideal $I_H$ of $L_R(E)$ generated by $\{p_v:v\in H\}$ is a basic ideal by \cite[Proposition~7.7]{Tomforde:JPAA11}.
When $E$ satisfies Condition~(K), the map $H\mapsto I_H$ is an isomorphism from the lattice of hereditary saturated subsets of $E^0$ onto the lattice of basic ideals of $L_R(E)$ by \cite[Corollary~7.18]{Tomforde:JPAA11}.  Theorem~\ref{thm:LPideals} addresses the non-basic ideal structure of $L_R(E)$ when $E$ satisfies (K).

\begin{thm}\label{thm:LPideals}
Let $E$ be a row-finite directed graph with no sources and let $R$ a commutative ring with identity.
Suppose that $E$ satisfies Condition~(K).
\begin{enumerate}
\item\label{it:Lideals} Suppose that $I$ is an ideal in $L_R(E)$. Then
    \[
        I= \lsp_R \{rs_{\lambda}s_{\mu^*}: rp_{s(\mu)} \in I\}.
    \]
\item\label{it:Llattice} Let $\mathcal{H}_E$ be the set of all saturated hereditary
    subsets of $E^0$, let $\mathcal{L}(R)$ be the set of ideals of $R$ and let
    $\mathcal{F}$ be the set of all functions $\pi:\mathcal{H}_E \to \mathcal{L}(R)$
    such that
    \[
    \pi\Big(\bigvee_{H \in \mathcal{A}}H\Big) = \bigcap_{H \in \mathcal{A}}\pi(H)
        \quad\text{ for all $\mathcal{A} \subseteq \mathcal{H}_E$.}
    \]
    Then the map $\Gamma:\mathcal{F} \to \mathcal{L}(L_R(E))$ given by
    \[
        \Gamma(\pi) = \lsp_R\{rs_{\mu}s_{\nu^*} :\text{ there exists }H \in \mathcal{H}_E \text{ such that }
        r \in \pi(H) \text{ and } s_{\mu}s_{\nu^*} \in I_H\}
    \]
    is a bijection.
\item\label{it:Lm&j} Let $\pi_1,\pi_2\in\mathcal{F}$. Then
    $\Gamma(\pi_1) \subseteq \Gamma(\pi_2)$ if and only if $\pi_1(H) \subseteq
    \pi_2(H)$ for all $H \in \mathcal{H}_E$.
\end{enumerate}
\end{thm}

Roughly, part \eqref{it:Lideals} of Theorem~\ref{thm:LPideals} comes from Proposition~\ref{prop:ideals},   
\eqref{it:Llattice} comes from Theorem~\ref{thm:bijection}, and \eqref{it:Lm&j} comes from Lemma~\ref{lem:containment} used in the proof of Theorem~\ref{thm:bijection}.  As we said above, the proof of Theorem~\ref{thm:LPideals}  follows from  the analogous  Theorem~\ref{thm:KPideals} for Kumjian--Pask algebras, which we state and prove below.  We now outline the background needed to state Theorem~\ref{thm:KPideals}.

For a positive integer $k$, the additive semigroup $\N^k$ can be viewed as a category with one object. Following Kumjian and Pask's \cite[Definitions~1.1]{KP}, a \emph{graph of rank $k$} or \emph{$k$-graph}
is a countable category $\Lambda = (\Lambda^0,\Lambda,r,s)$ together with a functor $d:\Lambda \to {\N}^k$, called
\emph{the degree map}, satisfying the following \emph{factorisation property}:
if $\lambda \in \Lambda$ and $d(\lambda) = m+n$ for some $m, n \in {\mathbb N}^k$, then there are unique $\mu,
\nu \in \Lambda$ such that $d(\mu) = m$, $d(\nu) = n$, and $\lambda = \mu\nu$. 

Let $\Lambda$ be a $k$-graph. 
We use the notational convention
whereby the juxtaposition $UV$ of subsets $U,V \subseteq \Lambda$ means $\{\mu\nu : \mu
\in U, \nu \in V, s(\mu) = r(\nu)\}$. If one of $U, V$ is a singleton, we typically drop
the braces from our notation; so for $v \in \Lambda^0$, the expression $v\Lambda$ means
the same as $\{v\}\Lambda$, namely $\{\lambda \in \Lambda : r(\lambda) = v\}$.

Following \cite{KP}, $\Lambda$ is \emph{row-finite} if
$v\Lambda^n$ is finite for every $v\in\Lambda^0$ and $n\in{\mathbb N}^k$; $\Lambda$ \emph{has no sources} if
$v\Lambda^n$ is nonempty for every $v\in\Lambda^0$ and $n\in{\mathbb N}^k$. In this paper we are only interested  in row-finite $k$-graphs with no sources. 

\begin{ex}\label{ex-k} 
Let $\Omega_k$ be the category with objects $\N^k$, morphisms $\{(p,q)\in \N^k\times\N^k : p\leq q\}$, domain and codomain maps given by $s(p,q)=q$ and $r(p,q)=p$ respectively, and composition given by $(p,q)(q,r)= (p,r)$. Define $d : \Omega_k \to \N^k$ by $d(p,q) := q-p$.  With this structure, $\Omega_k$ is a $k$-graph (where we identify $\Omega_k^0 = \{(p,p) : p \in \N^k\}$ with $\N^k$ via $(p,p)\mapsto p$).
\end{ex}

Following \cite{RSY-PEMS}, a subset $H$ of $\Lambda^0$ is \emph{hereditary} if $s(H\Lambda) \subseteq H$ and is
\emph{saturated} if $v \in H$ whenever $s(v\Lambda^n) \subseteq H$. Analogously to the definitions for directed graphs $E$ above, define
\[
\mathcal{H}_\Lambda := \{H \subseteq \Lambda^0 : H\text{ is saturated and hereditary}\}.
\]
Given a subset $\mathcal{A}$ of $\mathcal{H}_\Lambda$, let $\bigvee\mathcal{A}$ denote
the smallest element of $\mathcal{H}_\Lambda$ containing $\bigcup_{K \in \mathcal{A}} K$.
So
\[
\bigvee \mathcal{A} := \bigcap_{H \in \mathcal{H}_\Lambda,\; K \subseteq H\text{ for all }K\in\mathcal{A}} H.
\]
(We describe this $\bigvee$ operation explicitly in Lemma~\ref{lem:lattice}.)

The following is from \cite[\S2]{KP}. Let $\Omega_k$ be the $k$-graph of Example~\ref{ex-k}. An \emph{infinite path} in $\Lambda$ is a $k$-graph morphism $x:\Omega_k\to\Lambda$; the set of infinite paths is denoted by $\Lambda^\infty$.   We write $x(m)$ for the vertex $x(m,m)$. Then the range of an infinite path $x$ is the vertex $r(x):=x(0)$, and we write $v\Lambda^\infty:=r^{-1}(v)$.

 For $\lambda\in \Lambda$, set $Z(\lambda)=\{x\in\Lambda^\infty: x(0,d(\lambda))=\lambda\}$.  
Then  $\{Z(\lambda):\lambda\in\Lambda\}$ is a basis for a topology, and we equip $\Lambda^\infty$ with this topology.  Then $\Lambda^\infty$ is a totally disconnected, locally compact Hausdorff space, and each $Z(\lambda)$ is compact and open. For $p\in\N^k$ define $\sigma^p:\Lambda^\infty\to\Lambda^\infty$ by $\sigma^p(x)(m,n)=x(m+p, n+p)$.

By \cite[Definition~4.3]{KP}, a $k$-graph is \emph{aperiodic} if for every $v \in \Lambda^0$ there
exists $x \in Z(v)=v\Lambda^\infty$ such that
\begin{equation}\label{eq:aperiodic}
    \sigma^m(x) \not= \sigma^n(x)\text{ for all distinct $m,n \in \mathbb{N}^k$.}
\end{equation}
We say $\Lambda$ is \emph{strongly aperiodic} if for every saturated hereditary subset $H\neq \Lambda^0$ of $\Lambda$, the $k$-graph  $\Lambda\setminus H$ is  aperiodic. This is the analogue for $k$-graphs of Condition~(K) for directed graphs. (The terminology ``strongly aperiodic'' was coined in \cite[Definition~3.1]{KangPask:IJM14}, but the condition itself appeared earlier, for example, in \cite[Proposition~4.5]{RS-Israel}.)
We prove in Corollary~\ref{cor-se-iff-sa} below that $\Lambda$ is strongly aperiodic if and only if the graph groupoid of $\Lambda$ is strongly effective.

We refer to \cite[\S3]{ACaHR} for the definition of the Kumjian--Pask algebra path $\KP_R(\Lambda)$. We write $(p,s)$ for the universal generating Kumjian--Pask family in  $\KP_R(\Lambda)$. 
For $H\in \mathcal{H}_\Lambda$, the ideal $I_H$ of $\KP_R(E)$ generated by $\{p_v:v\in H\}$ is a basic ideal by \cite[Lemma~5.4]{ACaHR}.
When $\Lambda$ is strongly aperiodic, the map $H\mapsto I_H$ is an isomorphism from the lattice of saturated hereditary subsets of $\Lambda^0$ onto the lattice of basic ideals of $\KP_R(\Lambda)$ by \cite[Corollary~5.7]{ACaHR}.  We can now state our theorem for Kumjian--Pask algebras --- it looks very similar to Theorem~\ref{thm:LPideals}.

\begin{thm}\label{thm:KPideals}
Let $\Lambda$ be a row-finite $k$-graph with no sources and let $R$ a commutative ring with
identity. Suppose that $\Lambda$ is strongly aperiodic.
\begin{enumerate}
\item\label{it:KPideals} Suppose that $I$ is an ideal in $\KP_R(\Lambda)$. Then
    \[
        I= \lsp_R\{rs_{\lambda}s_{\mu^*}: rp_{s(\mu)} \in I\}.
    \]
\item\label{it:KPlattice} Let $\mathcal{H}_\Lambda$ be the set of all saturated
    hereditary subsets of $\Lambda^0$, let $\mathcal{L}(R)$ be the set of ideals of
    $R$ and let $\mathcal{F}$ be the set of all functions $\pi:\mathcal{H}_\Lambda
    \to \mathcal{L}(R)$ such that
    \begin{equation}\label{eq:2cain-KP}
    \pi\Big(\bigvee_{H \in \mathcal{A}}H\Big) = \bigcap_{H \in \mathcal{A}}\pi(H)
        \quad\text{ for all $\mathcal{A} \subseteq \mathcal{H}_\Lambda$.}
    \end{equation}
    Then the map $\Gamma:\mathcal{F} \to \mathcal{L}(\KP_R(\Lambda))$ given by
    \[
        \Gamma(\pi) = \lsp_R\{rs_{\mu}s_{\nu^*} : \text{ there exists }H \in \mathcal{H}_{\Lambda} \text{ such that }
        r \in \pi(H) \text{ and } s_{\mu}s_{\nu^*} \in I_H\}
    \]
    is a bijection.
\item\label{it:KPm&j} Let $\pi_1,\pi_2\in\mathcal{F}$. Then
    $\Gamma(\pi_1) \subseteq \Gamma(\pi_2)$ if and only if $\pi_1(H) \subseteq
    \pi_2(H)$ for all $H \in \mathcal{H}_\Lambda$.
\end{enumerate}
\end{thm}
To recover Theorem~\ref{thm:LPideals} from Theorem~\ref{thm:KPideals}, recall that for a
row-finite graph $E$ with no sources, the Leavitt path algebra $L_R(E)$ is canonically
isomorphic to the Kumjian--Pask algebra $\KP_R(E^*)$ where $E^*$ is the 
path-category of $E$ as in \cite[Example~1.3]{KP}.

Before starting the proof of Theorem~\ref{thm:KPideals}, we need to introduce the graph groupoid $G_\Lambda$ from \cite[Definition~2.7]{KP}. Let $\Lambda$ be a row-finite $k$-graph $\Lambda$ with no sources. Define
\[
    G_{\Lambda} := \{(x,l,y) \in\Lambda^\infty\times \Z^k\times\Lambda^\infty:\exists m,n\in\N^k \text{ such that }l=m-n\text{\ and\ }\sigma^m(x) = \sigma^n(y)\}.
\]
Then $G_{\Lambda}$ is a groupoid with composition and inverse given by
\[
    (x,l,y)(y, m, z) = (x, l+m, z)\quad\text{and}\quad (x,l,y)^{-1} = (y,-l,x).
\]
For $\mu, \nu \in \Lambda$ with $s(\mu) = s(\nu)$ set 
\[   
    Z(\mu,\nu) := \{(\mu z, d(\mu)-d(\nu), \nu z)  : z\in Z(s(\mu))\}
\]
Then $\{Z(\mu,\nu) :\mu,\nu\in \Lambda, s(\mu) = s(\nu)\}$ is a basis for a topology on $G_\Lambda$; we equip $G_\Lambda$ with this topology.  Then $G_{\Lambda}$ is an ample Hausdorff groupoid (see \cite[Proposition~2.8]{KP}). The unit space $G_{\Lambda}^{(0)}$  is $\{(x,0,x):x\in\Lambda^\infty\}$, which we identify with $\Lambda^\infty$; the identification takes $Z(\mu, \mu)$ to $Z(\mu)$.

Let $R$ be any commutative ring with identity. The Kumjian--Pask algebra KP$_R(\Lambda)$ is canonically isomorphic to 
the Steinberg algebra of  $A_R(G_{\Lambda})$. This is proved in \cite[Proposition~4.3]{CFST} when $R=\C$, for a directed graph  in \cite[Example~3.2]{CS2015}, and, most generally,  for a finitely aligned $k$-graph   in \cite[Proposition~5.4]{CP}. (A row-finite $k$-graph with no sources is finitely aligned.)

The next lemma establishes  the relationship between saturated hereditary
subsets of $\Lambda^0$ and open invariant subsets of $G_\Lambda^{(0)}$; we write
$\mathcal{O}_\Lambda$ for the latter.  This lemma is known but doesn't seem to be
recorded in the literature.

\begin{lemma}\label{lem:lattice}
Let $\Lambda$ be a row-finite $k$-graph with no sources. With notation as above, for
$\mathcal{A} \subseteq \mathcal{H}_\Lambda$, we have
\begin{equation}\label{eq:vee}
\textstyle \bigvee_{H \in \mathcal{A}} H
    = \{v \in \Lambda^0 : \text{ there exists } n \in \N^k\text{ such that }
            s(v\Lambda^n) \subseteq \bigcup_{H \in\mathcal{A}} H\}.
\end{equation}
The map $H \mapsto U_H$ from $\mathcal{H}_\Lambda$ to $\mathcal{O}_\Lambda$ given by
\[
U_H := \{x \in \Lambda^\infty : x(n) \in H\text{ for large }n \in \N^k\}
\]
is a bijection. For $H \in \mathcal{H}_\Lambda$, we have $H = \{v \in E^0 : Z(v)
\subseteq U_H\}$, and $U_{\bigvee_{H \in \mathcal{A}} H} = \bigcup_{H \in \mathcal{A}}
U_H$ for all $\mathcal{A} \subseteq \mathcal{H}_\Lambda$.
\end{lemma}

\begin{proof}
To establish~\eqref{eq:vee}, first take $v \in \Lambda^0$ and $n \in \N^k$ such that
$s(v\Lambda^n) \subseteq\bigcup_{H \in \mathcal{A}} H$. 
So $v$ belongs to every saturated hereditary set containing
$\bigcup_{H \in \mathcal{A}} H$, and therefore to $\bigvee_{H \in \mathcal{A}} H$. This
establishes ``$\supseteq$" in~\eqref{eq:vee}.

For the reverse containment, we show that any $v$ that does not belong to the right-hand
side of~\eqref{eq:vee} also does not belong to the left-hand side. Fix $v \in \Lambda^0$
such that $s(v\Lambda^n) \not\subseteq \bigcup_{H \in \mathcal{A}} H$ for all $n \in
\N^k$. For each $i\in\N$, let $n_i := (i, i, \dots, i) \in \N^k$. For each $i$, choose $\lambda_i\in v\Lambda^{n_i}$ such that $s(\lambda_i)\notin   \bigcup_{H \in \mathcal{A}} H$. Since $\Lambda$ has no sources, we can extend each $\lambda_i$ to an infinite path $x_i\in v\Lambda^\infty$ such that $x_i(n_i) \not\in \bigcup_{H \in \mathcal{A}} H$.  Since $v\Lambda^\infty=Z(v)$ is compact, there is a subsequence of $\{x_i\}$ converging to some $x \in Z(v)$. Since the $H \in \mathcal{A}$ are all hereditary, each $x_i(n_i)\notin \bigcup_{H \in \mathcal{A}} H$ implies $x_i(n)\notin \bigcup_{H \in \mathcal{A}} H$  for $n\leq n_i$. Since $\{n_i\}$ is a cofinal sequence,  it follows that $x(n) \not \in \bigcup_{H \in \mathcal{A}} H$ for all $n$.
Set \[K := \{w \in \Lambda^0 : w \Lambda x(n) =
\emptyset\text{ for all }n \in \N^k\}.\] Then $K$ is hereditary.

To see that $K$ is saturated, fix $w \in \Lambda^0$ and $n \in \mathbb{N}^k$ such that
$s(w\Lambda^n) \subseteq K$. Fix $\lambda \in w\Lambda$. For any $\alpha \in
s(\lambda)\Lambda^n$, by the factorisation property there exist $\mu \in w\Lambda^n$ and $\beta\in \Lambda$ such that $\lambda\alpha = \mu\beta$.
Now $s(\mu) \in s(w\Lambda^n)\subseteq K$. Thus $s(\beta) \neq x(m)$ for all $m \in \N$.   In particular, $\alpha
\neq x(p, p+n)$ for any $p$. Thus $s(\lambda)\neq x(p)$ for all $p$, and hence $w \in K$.

We have $v \not\in K$ by construction of $K$. Fix $u \in H
\in \mathcal{A}$. If $u\Lambda x(n)$ were nonempty for some $n$, we would have $x(n) \in
H$ because $H$ is hereditary, and this is impossible by construction of $x$. So $u \in
K$. That is, $K$ is a saturated hereditary set containing $\bigcup_{H \in \mathcal{A}} H$
and not containing $v$, and it follows from the definition of $\bigvee_{H \in
\mathcal{A}} H$ that $v \not\in \bigvee_{H \in \mathcal{A}} H$ as required.  This gives~\eqref{eq:vee}.

To see that $H \mapsto U_H$ is injective, suppose that $U_{H_1} = U_{H_2}$. By symmetry,
we just have to show that $H_1 \subseteq H_2$. Let $v \in H_1$. Then $Z(v) \subseteq
U_{H_1} = U_{H_2}$. So for each $x \in Z(v)$, there exists $n_x\in\N^k$ such that $x(n_x) \in
H_2$. The sets $Z(x(0, n_x))$ cover $Z(v)$, and so there is a finite $F \subseteq Z(v)$
such that $\{Z(x(0, n_x)) : x \in F\}$ covers $Z(v)$. Take $N := \bigvee_{x \in F} n_x$. Let $\lambda \in
v\Lambda^N$. Then $\lambda\in Z(v)$ implies $\lambda\in Z(x(0, n_x))$ for some  $x \in F$. 
Since $x(n_x)\in H_2$ and $H_2$ is hereditary, we have  $s(\lambda) \in H_2$. Thus $s(v\Lambda^N)\subseteq H_2$. Since $H_2$ is saturated as well, we have $v \in H_2$.

To see that $H \mapsto U_H$  is surjective, fix an open invariant $U \subseteq G_\Lambda^{(0)}$. Put \[H(U) := \{v
\in \Lambda^0 : Z(v) \subseteq U\}.\] 
We claim that $H(U)$ is saturated and hereditary, and that $U_{H(U)}=U$. To see that $H(U)$ is hereditary, let $w\in H(U)$ and $\lambda\in w\Lambda$. Let $x\in  Z(s(\lambda))$. Then $(\lambda x, d(\lambda), x)\in G_\Lambda$ and $\lambda x=r(\lambda x, d(\lambda), x)\in Z(w)\subseteq U$. Since $U$ is invariant, $x=s(\lambda x, d(\lambda), x)\in U$ as well. Thus $Z(s(\lambda))\subseteq U$, and hence $s(\lambda)\in H(U)$.
So $H(U)$ is hereditary.  To see that $H(U)$ is saturated, let  $n
\in \N^k$ and $w \in \Lambda^0$, and suppose  that $s(w\Lambda^n) \subseteq H(U)$. If $\lambda\in w\Lambda^n$, then $s(\lambda)\in H(U)$ implies $\lambda x\in U$ by invariance of $U$.
Thus \[Z(w) =\bigcup_{\lambda\in w\Lambda^n} Z(\lambda)=
\bigcup_{\lambda \in w\Lambda^n} \{r(\lambda x, d(\lambda), x) : x \in Z(s(\lambda))\}
\subseteq U.\]  So $w \in H(U)$. Thus $H(U)$ is saturated. 

To see  that $U_{H(U)} = U$, first let $x \in U$.  Since $U$ is open there exists $\lambda\in \Lambda$ such that $x\in Z(\lambda)\subseteq U$. Since $U$ is  invariant and $(x, d(\lambda), \sigma^{d(\lambda)}(x))\in G_\Lambda$ we see that $Z(s(\lambda)) \subseteq U$ as well. Thus $s(\lambda) \in H(U)$.
Since $H(U)$ is hereditary,  $x(n) \in H(U)$ for all $n \ge d(\lambda)$, and
hence $x \in U_{H(U)}$. 

Second, let $x \in U_{H(U)}$. Then there exists $n$ such that $x(n) \in
H(U)$. Then $Z(x(n)) \subseteq U$, and hence $\sigma^n(x) \in U$. Since $U$ is invariant, it
follows that $x = r(x, n, \sigma^n(x)) \in U$ as well. Thus $U_{H(U)}=U$, and  $H\mapsto U_H$ is surjective.

That $H = \{v \in \Lambda^0 : Z(v)
\subseteq U_H\}$ follows quickly: given $H$, we have  $U_{H(U_H)} = U_H$,
and since  $H \mapsto U_H$ is injective, we deduce that $H =
H(U_H)$, which is $\{v \in E^0 : Z(v) \subseteq U_H\}$ by definition. 

It remains to check compatibility of $\bigvee$ with $\bigcup$. Fix $\mathcal{A} \subseteq
\mathcal{H}_\Lambda$. First suppose that $x \in U_{\bigvee_{H \in \mathcal{A}} H}$. Then
there exists $n$ such that $x(n) \in \bigvee_{H \in \mathcal{A}} H$.
Equation~\eqref{eq:vee} shows that there exists $m$ such that $x(n+m) \in \bigcup_{H \in
\mathcal{A}} H$, so we may fix $H \in \mathcal{A}$ with $x(n+m) \in H$. Since
$H$ is hereditary, we have $x(p) \in H$ for large $p$, giving $x \in U_H \subseteq
\bigcup_{H \in \mathcal{A}} U_H$. Second,  suppose that $x \in \bigcup_{H \in \mathcal{A}}
U_H$. Then $x \in U_H$ for some $H \in \mathcal{A}$, and since $U_H \subseteq
U_{\bigvee_{H \in \mathcal{A}} H}$, we deduce that $x \in U_{\bigvee_{H \in \mathcal{A}}
H}$. Thus $U_{\bigvee_{H \in \mathcal{A}} H} = \bigcup_{H \in \mathcal{A}}
U_H$.
\end{proof}

\begin{cor}\label{cor-se-iff-sa}
\begin{enumerate} \item\label{1cor-se-iff-sa} Let $\Lambda$ be a row-finite $k$-graph with no sources. Then $\Lambda$ is strongly aperiodic if and only if the $k$-graph groupoid $G_\Lambda$ is strongly effective.
\item\label{2cor-se-iff-sa}  Let $E$ be a row-finite directed graph with no sources. Then $E$ satisfies Condition~(K)  if and only if the graph groupoid $G_E$ is strongly effective.
\end{enumerate}
\end{cor}

\begin{proof} For \eqref{1cor-se-iff-sa}, fix a saturated hereditary subset $H$ of $\Lambda^0$ such that $H\neq \Lambda^0$. Let $U_H$ be the corresponding open $G$-invariant subset of $G_\Lambda^{(0)}$ of Lemma~\ref{lem:lattice}. The subcategory $\Lambda\setminus\Lambda H:=(\Lambda^0\setminus H, s^{-1}(\Lambda^0\setminus H), r, s)$ of $\Lambda$ is a $k$-graph.  For $x\in \Lambda^\infty$, we have $x\in (\Lambda\setminus\Lambda H)^\infty$ if and only if $x\not\in U_H$. Then $V_H:=G_\Lambda^{(0)}\setminus U_H$ is a nonempty closed invariant subset of $\go_\Lambda$.  We have 
\[
(G_\Lambda)_{V_H}=\{(x,l,y)\in G_\Lambda: x,y\in V_H\}=\{(x,l,y)\in G_\Lambda: x,y\notin U_H\}=G_{\Lambda\setminus \Lambda H}
\]
Thus $(G_\Lambda)_{V_H}$ is effective if and only if $G_{\Lambda\setminus \Lambda H}$ is effective. Since $G_{\Lambda \setminus \Lambda H}$ is second countable, it is effective
if and only if it is topologically principal by \cite[Proposition~3.6]{Renault:IMSB08}.
Proposition~4.5 of \cite{KP} says that $G_{\Lambda \setminus \Lambda H}$ is
topologically principal if and only if $\Lambda \setminus \Lambda H$ is aperiodic.

Since $H\mapsto U_H$ is a bijection by Lemma~\ref{lem:lattice}, it follows that  $G_\Lambda$ is strongly effective if and only if $\Lambda$ is strongly aperiodic. This gives \eqref{1cor-se-iff-sa}.

For \eqref{2cor-se-iff-sa}, let $F$ be a row-finite directed graph with no sources and let $F^*$ be the path category as in \cite[Example~1.3]{KP}. Then $F^*$ is a $1$-graph.  By \cite[Lemma~4.6]{ACaHR}, $F^*$ is aperiodic if and only if $F$ satisfies Condition~(L). Thus \eqref{2cor-se-iff-sa} follows as in the proof of \eqref{1cor-se-iff-sa}. 
\end{proof}

\begin{proof}[Proof of Theorem~\ref{thm:KPideals}]
Since $\Lambda$ is strongly aperiodic, $G_\Lambda$ is strongly effective by Corollary~\ref{cor-se-iff-sa}. Thus Proposition~\ref{prop:ideals} and Theorem~\ref{thm:bijection} apply to $G_\Lambda$.

(\ref{it:KPideals}) Every compact open bisection in $G_\Lambda$ is a
finite union of basic compact open bisections $Z(\lambda,\mu)$. So in $A_R(G_\Lambda)$,
we have
\begin{align*}
I=\lsp_R\{r1_B : B&\text{ is a compact open bisection with } r 1_{s(B)} \in I\} \\
    &= \lsp_R\{r 1_{Z(\lambda,\mu)} : \lambda,\mu\in \Lambda, s(\lambda)=s(\mu), r1_{Z(\mu)} \in I\}.
\end{align*}
We have
\[
r1_{Z(\mu)} = 1_{Z(\mu,s(\mu))}\big(r 1_{Z(s(\mu))}\big)1_{Z(s(\mu),\mu)}
\quad\text{ and }\quad
r1_{Z(s(\mu))} = 1_{Z(s(\mu),\mu)}\big(r 1_{Z(\mu)}\big)1_{Z(\mu, s(\mu))},
\]
and hence $r1_{Z(\mu)} \in I$ if and only if $r1_{Z(s(\mu))} \in I$. Hence
\begin{align*}
\lsp_R\{r1_B : B&\text{ is a compact open bisection with } r 1_{s(B)} \in I\} \\
    &= \lsp_R\{r 1_{Z(\lambda,\mu)} : \lambda,\mu\in \Lambda, s(\lambda)=s(\mu), r1_{Z(s(\mu))} \in I\}.
\end{align*}
The canonical isomorphism of  $A_R(G_\Lambda)$ onto $\KP_R(\Lambda)$ of \cite[Proposition~5.4]{CP} carries  $1_{Z(\lambda,\mu)}$ to
$s_\lambda s_{\mu^*}$ and $1_{Z(v)}$ to $p_v$. Thus $I=\lsp_R\{rs_{\lambda}s_{\mu^*}: rp_{s(\mu)} \in I\}$
by Proposition~\ref{prop:ideals}.

(\ref{it:KPlattice}) 
Composition with the bijection $U\mapsto H_U$ from $\mathcal{O}_\Lambda$ to $\mathcal{H}_\Lambda$ of 
Lemma~\ref{lem:lattice} carries the functions $\pi :
\mathcal{H}_\Lambda \to \mathcal{L}(R)$ satisfying \eqref{eq:2cain-KP} to
functions from $\mathcal{O}_\Lambda$ to $\mathcal{L}(R)$ satisfying \eqref{eq:2cain} in
Theorem~\ref{thm:bijection}. Thus it follows from Theorem~\ref{thm:bijection} that $\Gamma:\mathcal{F}\to\mathcal{L}(R)$ is a bijection.

The argument of part~(\ref{it:KPideals}) shows that the
isomorphism $\KP_R(\Lambda)$ onto $A_R(G_\Lambda)$ carries the ideal $I_H$ generated by the $p_v$ with $v\in H$
 to the ideal $I_{U_H}$, and hence $\Gamma(\pi)$ has the form claimed.

(\ref{it:KPm&j}) This follows from Lemma~\ref{lem:containment} because $H\mapsto U_H$ preserves containment.
\end{proof}

We conclude by applying our results to two illustrative examples of Leavitt path
algebras.

\begin{ex}
Consider the directed graph $E$ pictured below.
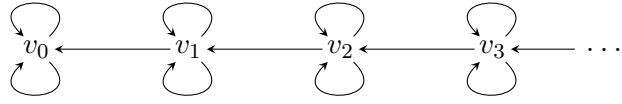
\[
\begin{tikzpicture}
    \node[circle, inner sep=0.5pt] (v0) at (0,0) {\small$v_0$};
    \node[circle, inner sep=0.5pt] (v1) at (2,0) {\small$v_1$};
    \node[circle, inner sep=0.5pt] (v2) at (4,0) {\small$v_2$};
    \node[circle, inner sep=0.5pt] (v3) at (6,0) {\small$v_3$};
    \node (dots) at (7.5,0) {\dots};
    \draw[-stealth] (v0) .. controls +(0.75, 0.75) and +(-0.75,0.75) .. (v0);
    \draw[-stealth] (v0) .. controls +(0.75, -0.75) and +(-0.75,-0.75) .. (v0);
    \draw[-stealth] (v1) .. controls +(0.75, 0.75) and +(-0.75,0.75) .. (v1);
    \draw[-stealth] (v1) .. controls +(0.75, -0.75) and +(-0.75,-0.75) .. (v1);
    \draw[-stealth] (v2) .. controls +(0.75, 0.75) and +(-0.75,0.75) .. (v2);
    \draw[-stealth] (v2) .. controls +(0.75, -0.75) and +(-0.75,-0.75) .. (v2);
    \draw[-stealth] (v3) .. controls +(0.75, 0.75) and +(-0.75,0.75) .. (v3);
    \draw[-stealth] (v3) .. controls +(0.75, -0.75) and +(-0.75,-0.75) .. (v3);
    \draw[-stealth] (v1)--(v0);
    \draw[-stealth] (v2)--(v1);
    \draw[-stealth] (v3)--(v2);
    \draw[-stealth] (dots)--(v3);
\end{tikzpicture}
\]
This $E$ satisfies Condition~(K) because every vertex has two loops, and it has a
linear lattice $\mathcal{H}_E = \{H_n : n \in\N\}$ of saturated hereditary sets
$H_n = \{v_n, v_{n+1}, v_{n+2}, \dots\}$.

Consider the ring $R = \mathbb{Z}$, which has nonzero ideals $\{m\mathbb{Z} : m\in\N\setminus\{0\}\}$. As a notational convenience, we write $\infty\mathbb{Z} := \{0\}$, the trivial
ideal. So we may identify the set of functions $\pi : \mathcal{H}_E \to \mathcal{L}(R)$
with the set of all functions $\pi : \mathbb{N} \to \mathbb{N} \cup\{\infty\}$. Given a
subset $\mathcal{A} \subseteq \mathbb{N}$, we have $\bigvee_{n \in \mathcal{A}} H_n = H_{\min \mathcal{A}}$, and so a
given $\pi : \mathbb{N} \to \mathbb{N} \cup\{\infty\}$ belongs to $\mathcal{F}$ if
$\pi(\min \mathcal{A}) = \operatorname{lcm}\{\pi(n) : n \in \mathcal{A}\}$ for all $\mathcal{A} \subseteq \mathbb{N}$.
This is equivalent to the condition that $\pi(n+1) \mid \pi(n)$ for all $n$ (with the
convention that $n \mid \infty$ for every $n \in \mathbb{N} \cup \{\infty\})$. So
$\mathcal{F}$ consists of functions $\pi : \mathbb{N} \to \mathbb{N} \cup \{\infty\}$
such that $\pi(n+1) \mid \pi(n)$ for all $n$. Given such a function $\pi$, the
corresponding ideal $\Gamma(\pi)$ of $A_\mathbb{Z}(E)$ is
\[
\Gamma(\pi) = \lsp_\mathbb{Z}\{r s_\mu s^*_\nu : n \in \mathbb{N}, \mu,\nu \in E^*v_n\text{ and } r \in \pi(n)\mathbb{Z}\}.
\]
We have $\Gamma(\pi) \subseteq \Gamma(\pi')$ if and only if $\pi'(n) \mid \pi(n)$ for all
$n$. Theorem~\ref{thm:LPideals} shows that this completely describes all the ideals of
$L_\mathbb{Z}(E)$.
\end{ex}

\begin{ex}\label{ex:tree}
Consider the directed graph $E$ pictured below.
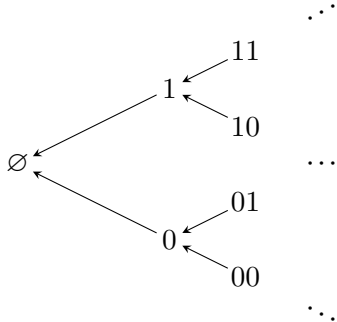
\[
\begin{tikzpicture}[yscale=0.5]
    \node[circle, inner sep=0.5pt] (e) at (0,0) {\small$\varnothing$};
    \node[circle, inner sep=0.5pt] (1) at (2,2) {\small$1$};
    \node[circle, inner sep=0.5pt] (0) at (2,-2) {\small$0$};
    \node[circle, inner sep=0.5pt] (11) at (3,3) {\small$11$};
    \node[circle, inner sep=0.5pt] (10) at (3,1) {\small$10$};
    \node[circle, inner sep=0.5pt] (01) at (3,-1) {\small$01$};
    \node[circle, inner sep=0.5pt] (00) at (3,-3) {\small$00$};
    \draw[-stealth] (0)--(e);
    \draw[-stealth] (1)--(e);
    \draw[-stealth] (00)--(0);
    \draw[-stealth] (01)--(0);
    \draw[-stealth] (10)--(1);
    \draw[-stealth] (11)--(1);
    \node at (3.85,-3.85) {$\cdot$};
    \node at (4,-4) {$\cdot$};
    \node at (4.15,-4.15) {$\cdot$};
    \node at (3.85,0) {$\cdot$};
    \node at (4,0) {$\cdot$};
    \node at (4.15,0) {$\cdot$};
    \node at (3.85,3.85) {$\cdot$};
    \node at (4,4) {$\cdot$};
    \node at (4.15,4.15) {$\cdot$};
\end{tikzpicture}
\]
To describe the ideals in $L_\mathbb{Z}(E)$ for this example, it is easiest to apply the
description given in Theorem~\ref{thm:alt description}. For this, observe that the
infinite paths in $E$, which are the units of the associated groupoid $G_E$,
can be
identified with pairs $(\omega, x)$ consisting of a finite word $\omega \in \{0,1\}^*$
and an infinite word $x \in \{0,1\}^\infty$  (when thinking of $\omega$ and $x$ as paths,  $\omega$ corresponds to the unique finite path to the root $\varnothing$ of the tree $E$ from the range of the infinite path $x$).

The graph groupoid $G_E$ then consists of
triples of the form $\big((\omega, x), p-q, (\omega', y)\big)$ such that $p + |\omega| =
q + |\omega'|$ and $x_{p+k} = y_{q+k}$ for all $k$, and from this it is easy to see that
every orbit  of $G_E$ intersects exactly once with the set $\{\varnothing\} \times
\{0,1\}^\infty$ of infinite paths with range $\varnothing$. So the $G_E$-invariant
functions $\rho : G_E^{(0)} \to \mathcal{L}(\mathbb{Z})$ are in bijective correspondence
with functions $\rho_0 : \{0,1\}^\infty \to \mathcal{L}(\mathbb{Z})$; specifically,
$\rho_0(x) = \rho\big((\varnothing, x)\big)$ and $\rho\big((\omega, x)\big) =
\rho_0(\omega x)$. Moreover, $\rho$ is continuous with respect to the  topology on
$\mathcal{L}(\mathbb{Z})$ described just after Example~\ref{ex:twopoints} if and only if $\rho_0$ is continuous with respect to the same topology on $\mathcal{L}(\Z)$ and the product topology on $\{0,1\}^\infty$. So the
assignment $\rho \mapsto \rho_0$ restricts to a bijection between the set $\mathcal{F}'$
of Theorem~\ref{thm:alt description} and the set of continuous functions from
$\{0,1\}^\infty$  (under the product topology) to $\mathcal{L}(\mathbb{Z})$.

To describe the topology on $\mathcal{L}(\mathbb{Z})$, observe that for a finite set $F
\subseteq \mathbb{Z}$, the corresponding open set $Z(F) \subseteq
\mathcal{L}(\mathbb{Z})$ is the set $\{n\mathbb{Z} : n \mid \gcd(F)\}$. Identifying
$\mathcal{L}(\mathbb{Z})$ with $\mathbb{N} \cup \{\infty\}$ as in the previous example,
we see that the open sets in $\mathbb{N} \cup \{\infty\}$ are the sets $\{n : n \mid N\}$
indexed by $N \in \mathbb{N} \cup \{\infty\}$. So a function $\rho_0 : \{0,1\} \to
\mathbb{N} \cup \{\infty\}$ is continuous if and only if whenever $x_j \to x$ in the
product topology on $\{0,1\}^\infty$ we have $\rho_0(x_j) \mid \rho_0(x)$ for large $j$.
For $\nu\in E^*$, write $\omega_\nu$ for the unique element of $\{0,1\}^\infty$ that corresponds to the path to the root $\varnothing$ from $s(\nu)$. By Theorem~\ref{thm:alt description}, the ideal corresponding to such a function $\rho_0$ is
\[
\Gamma'(\rho_0) = \lsp_{\mathbb{Z}}\{n s_\mu s_{\nu^*} :  \mu,\nu \in E^*s(\nu), n \in \rho_0(\omega_\nu x)\text{ for
all }x \in \{0,1\}^\infty\},
\]
and these are all the ideals of $L_\mathbb{Z}(E)$. Moreover, $\Gamma'(\rho_0) \subseteq
\Gamma'(\tau_0)$ if and only if $\tau_0(x) \mid \rho_0(x)$ for all $x \in
\{0,1\}^\infty$.
\end{ex}

\begin{rmk}
In the preceding example we argued directly to prove that we could reduce the problem of
describing $\mathcal{F}'$ to that of describing the collection of continuous functions
$\rho_0 : \{0,1\}^\infty \to \mathcal{L}(\mathbb{Z})$; but as an alternative, we could
have used the results of \cite{CS2015}. The set $X := \{\varnothing\} \times
\{0,1\}^\infty \cong \{0,1\}^\infty$ is compact open and intersects every $G_E$-orbit.
Hence the restriction $H$ of $G_E$ to this subset of the unit space is equivalent, in the
sense of Renault, to $G_E$ by \cite[Lemma~6.1]{CS2015}. So \cite[Theorem~5.1]{CS2015}
implies that $A_{\mathbb{Z}}(G_E)$ is Morita equivalent to $A_{\mathbb{Z}}(H)$, and hence
the ideals of the former are in bijection with the ideals of the latter. Since $X$
intersects every $G_E$-orbit exactly once, $H = H^{(0)}$ is just a copy of the
topological space $X$, so $H$-equivariance of a function $\rho_0 : H^{(0)} \to
\mathcal{L}(\mathbb{Z})$ is a vacuous requirement, and we deduce, once again, that the
ideals of $L_{\mathbb{Z}}(E)$ are in bijection with the continuous functions from
$\{0,1\}^\infty$ to $\mathcal{L}(\mathbb{Z})$.
\end{rmk}

\end{document}